\documentclass[a4paper,11pt]{article}
\usepackage[english]{babel}
\usepackage[T1]{fontenc}
\usepackage[utf8]{inputenc}
\usepackage{lmodern}
\usepackage{microtype}
\usepackage{hyperref}

\usepackage{amsmath,amsfonts,amssymb,amscd,amsthm,xspace,mathtools}
\usepackage{dsfont}
\usepackage{calrsfs}
\theoremstyle{plain}

\newtheorem{theorem}{Theorem}[section]
\newtheorem{thmi}{Theorem}
\newtheorem{corollary}[theorem]{Corollary}
\newtheorem{lemma}[theorem]{Lemma}
\newtheorem{proposition}[theorem]{Proposition}

\theoremstyle{definition}
\newtheorem{definition}[theorem]{Definition}
\newtheorem*{defi}{Definition}
\newtheorem{claim}{Claim}[theorem]
\newenvironment{claimproof}[1]{\bigskip\textit{Proof.}\space#1}{\hfill $\blacksquare$}
\theoremstyle{remark}

\usepackage[centerlast,small,sc]{caption}
\newenvironment{pushright}{\begin{itemize}\item[\hspace{12pt}]}{\end{itemize}}
\DeclareMathAlphabet{\pazocal}{OMS}{zplm}{m}{n}
\newcommand{\1}{\mathds 1}
\newcommand{\fA}{\mathfrak A}

\newcommand{\pA}{\pazocal A}
\newcommand{\B}{\mathrm B}

\newcommand{\pB}{\pazocal B}
\newcommand{\pC}{\pazocal C}
\newcommand{\mC}{\mathcal C}
\newcommand{\pD}{\pazocal D}
\newcommand{\rE}{\mathrm E}

\newcommand{\mG}{\mathcal G}
\newcommand{\pG}{\pazocal G}

\newcommand{\pH}{\pazocal H}

\newcommand{\pI}{\pazocal I}

\newcommand{\pL}{\pazocal L}
\newcommand{\rL}{\mathrm L}
\newcommand{\pM}{\pazocal M}
\newcommand{\N}{\mathbb N}
\newcommand{\pN}{\pazocal N}
\newcommand{\rN}{\mathrm N}
\newcommand{\cP}{\mathbb P}
\newcommand{\fP}{\mathfrak P}
\newcommand{\mP}{\mathcal P}

\newcommand{\mR}{\mathcal R}

\newcommand{\pU}{\pazocal U}
\newcommand{\rV}{\mathrm V}
\newcommand{\Z}{\mathbb Z}

\newcommand{\pZ}{\pazocal Z}
\newcommand{\AMA}{\mathrm{AMA}}
\newcommand{\inc}{\mathrm{inc}}

\usepackage{tikz-cd}
\usepackage{graphicx}
\usepackage{epstopdf}
\usepackage[scriptsize]{subfigure}
\usepackage{booktabs}
\usepackage{rotating}
\usepackage{listings}
\usepackage{fancyhdr}
\def\Ind#1#2{#1\setbox0=\hbox{$#1x$}\kern\wd0\hbox to 
0pt{\hss$#1\mid$\hss}
\lower.9\ht0\hbox to 0pt{\hss$#1\smile$\hss}\kern\wd0}
\def\ind{\mathop{\mathpalette\Ind{}}}
\def\Notind#1#2{#1\setbox0=\hbox{$#1x$}\kern\wd0\hbox to 
0pt{\mathchardef
\nn="3236\hss$#1\nn$\kern1.4\wd0\hss}\hbox to 
0pt{\hss$#1\mid$\hss}\lower.9\ht0
\hbox to 0pt{\hss$#1\smile$\hss}\kern\wd0}

\def\act#1{\overset{#1}{\curvearrowright}}

\begin{document}

\pagestyle{empty}

\title{Hyperfinite measure-preserving actions of countable groups and their model theory}
\author{Alice Giraud}

\date{October 17, 2019}

\maketitle

\begin{abstract}
We give a shorter proof of a theorem of G.~Elek stating that two hyperfinite measure-preserving actions of a countable group on standard probability spaces are approximately conjugate if and only if they have the same invariant random subgroup.

We then use this theorem to study model theory of hyperfinite measure-preserving actions of countable groups on probability spaces. This work generalizes the model-theoretic study of automorphisms of probability spaces conducted by I.~Ben~Yaacov, A.~Berenstein, C.~W.~Henson and A.~Usvyatsov.
\end{abstract}

\pagestyle{headings}

\tableofcontents

\clearpage


\section{Introduction}

Classical ergodic theory consists of the study of probability measure-preserving (pmp in short) transformations of a probability space. A \textbf{pmp transformation} $T$ of a probability space $(X,\mu)$ is a bimeasurable permutation of $X$ such that for all measurable subsets $A$ of $X$, $\mu(T^{-1} A) = \mu(A)$. It is called \textbf{ergodic} if any $T$-invariant subset of $X$ is either null or conull, and it is called \textbf{aperiodic} if almost every $T$-orbit is infinite. In the case of a single transformation $T$ of an atomless probability space, it is well-known that ergodicity implies aperiodicity. For now, we restrict ourselves to \textbf{standard probability spaces}, that is probability spaces that are isomorphic to the interval $[0,1]$ equipped with the Lebesgue measure.

Two pmp transformations $T$ and $T'$ are said to be \textbf{conjugate}, or sometimes \textbf{isomorphic}, if there is a third pmp transformation $S$ such that up to a null set, $T' = S T S^{-1}$. One of the main goals of ergodic theory is to understand the conjugacy relation on pmp transformations, particularly on the set of ergodic pmp transformations. Conjugacy is completely understood in some specific cases, for example, entropy is a complete invariant of conjugacy for Bernoulli shifts \cite{EntropyBernoulliShifts} and spectrum is a complete invariant of conjugacy for compact transformations. However, in general, conjugacy is a very complicated relation as shown in \cite{TurbulenceConjugacy} and \cite{FRW}.
\\

In this paper we study the simpler relation of approximate conjugacy. Two pmp transformations $T$ and $T'$ of $(X,\mu)$ are said to be \textbf{approximately conjugate} if for all $\varepsilon >0$ there is a third pmp transformation $S$ of $(X,\mu)$ such that $T' = STS^{-1}$ up to a set of measure at most $\varepsilon$. It is a well-known consequence of Rokhlin Lemma that any two aperiodic pmp transformations of standard probability spaces are approximately conjugate \cite[Thm. 2.4]{kechrisGlobalAspectsErgodic2010}. We thus focus on understanding the approximate conjugacy relation for general pmp actions of countable discrete groups rather than single pmp transformations, which correspond to $\Z$-actions.

A \textbf{pmp action} of a countable group $\Gamma$ on a probability space $(X,\mu)$ is an action of $\Gamma$ on $X$ by pmp transformations. For a pmp action $\Gamma \act{\alpha} (X,\mu)$ and $\gamma \in \Gamma$, we let $\gamma^\alpha$ denote the pmp transformation associated to $\gamma$ in the action $\alpha$. Two pmp actions $\alpha$ and $\beta$ of a countable group $\Gamma$ are \textbf{conjugate} if there is a pmp transformation $S$ such that $S^{-1} \gamma^\alpha S = \gamma^\beta$ for all $\gamma \in \Gamma$. We say that $\alpha$ is a \textbf{factor} of $\beta$, denoted by $\alpha \sqsubseteq \beta$ if there is a measure-preserving map $S : X \rightarrow X$ such that $\gamma^\alpha S = S \gamma^\beta$ for every $\gamma \in \Gamma$.

We say that $\alpha$ and $\beta$ are \textbf{approximately conjugate} if for every finite $F \subseteq \Gamma$ and every $\varepsilon >0$, there exists a pmp transformation $S$ of $X$ such that
$$\mu \left( \{ x \in X : \exists \gamma \in F,\ \gamma^\beta x \neq S \gamma^\alpha S^{-1} x \} \right) < \varepsilon.$$

This notion of approximate conjugacy comes from the study of the spaces $\mathrm{Aut}(X,\mu)$ and $A(\Gamma,X,\mu)$ of pmp transformations of $(X,\mu)$ and of pmp actions of $\Gamma$ on $(X,\mu)$, respectively.

The space $\mathrm{Aut}(X,\mu)$ can be equipped with two topologies: the weak and the uniform topology (see \cite{kechrisGlobalAspectsErgodic2010} for definitions).
Two pmp transformations $T$ and $S$ are called \textbf{weakly equivalent} if $\overline{[T]}^w = \overline{[S]}^w$, where $[T]$ is the conjugacy class of $T$, and $\overline{A}^w$ denotes the closure of $A$ in the weak topology. Then, the space of actions can be seen as a closed subspace of $\mathrm{Aut}(X,\mu)^\Gamma$ equipped with either product topology, and this induces two topologies on $A(\Gamma,X,\mu)$, that we respectively call again the weak and the uniform topology. In the same fashion as for transformations, we say that two actions $\alpha$ and $\beta$ are weakly equivalent if $\overline{[\alpha]}^w = \overline{[\beta]}^w$.

Now approximate conjugacy is the uniform counterpart of weak equivalence, that is, two pmp actions $\alpha$ and $\beta$ are approximately conjugate if and only if $\overline{[\alpha]}^u = \overline{[\beta]}^u$, where $\overline{A}^u$ is the uniform closure of $A$. The study of approximate conjugacy in the present paper was mostly motivated by similar results obtained for weak equivalence by R.~Tucker-Drob in \cite{tucker-drobWeakEquivalenceNonclassifiability2015}.
\\

The first obstacle to approximate conjugacy is freeness : a pmp action of $\Gamma$ is \textbf{free} if the set of fixed points of any nontrivial element of $\Gamma$ is null. For $\Z$-actions, freeness corresponds to aperiodicity. It is easy to see that approximate conjugacy preserves the freeness of the actions, and that the trivial action is only approximately conjugate with itself.

In fact, we have a better result. For a pmp action $\Gamma \act{\alpha} (X,\mu)$, the pushforward of the measure $\mu$ by the stabilizer application $x \in X \mapsto \mathrm{Stab}^\alpha(x)$ gives a measure $\theta_\alpha$ on the space of subgroups of $\Gamma$. We call this measure the Invariant Random Subgroup (IRS in short, see \cite{abertKestenTheoremInvariant2014}) of the action $\alpha$. Then it is not hard to see that the IRS is an invariant of approximate conjugacy. Moreover, free actions correspond to the case where the IRS is the Dirac measure on the trivial subgroup $\delta_{\{e\}}$ and the trivial action corresponds to the case where the IRS is $\delta_\Gamma$.
\\

In this paper we work with hyperfinite actions, which are defined as follows:

\begin{defi}
A pmp action $\Gamma \act{} (X,\mu)$ is said to be \textbf{hyperfinite} if for any finite subset $S$ of $\Gamma$ and any $\varepsilon>0$, there exists a finite group $G$ acting in a measure-preserving way on $(X,\mu)$ such that
$$\mu \left( \{ x \in X : S \cdot x \subseteq G \cdot x \} \right) > 1 - \varepsilon.$$
\end{defi}

It is a theorem of D.~S.~Ornstein and B.~Weiss \cite{OrnsteinWeiss} that pmp actions of amenable groups are hyperfinite.
\\

In general, we have the following implications:
$$\mathrm{approximate\ conjugacy} \Longrightarrow \mathrm{weak\ equivalence} \Longrightarrow \mathrm{same\ IRS}.$$
In the most general context, the IRS of an action is not a complete invariant of approximate conjugacy. However, G.~Elek proved that when restricted to hyperfinite actions, it is:

\begin{thmi}[G.Elek, {\cite[Thm. 9]{elekFiniteGraphsAmenability2012}}] \label{Elek} \label{approximate conjugacy for pmp actions with a given hyperfinite irs}
Let $\alpha$ and $\beta$ be two pmp hyperfinite actions of a group $\Gamma$ on a standard probability space such that $\theta_\alpha = \theta_\beta$. Then $\alpha$ and $\beta$ are approximately conjugate.
\end{thmi}

This theorem thus generalizes the consequence \cite[Thm. 2.4]{kechrisGlobalAspectsErgodic2010} of Rokhlin Lemma, which can be obtained by taking $\Gamma = \Z$ and $\theta_\alpha = \theta_\beta = \delta_{\{e\}}$.

In this paper, we give a shorter proof of this theorem, first by considering the critical case of actions which are factors one of another and then using a confluence argument to conclude in the general case. Moreover, when one of the actions is a factor of the other, we add a slight improvement to the theorem by requiring that the pmp transformations witnessing approximate conjugacy stabilize some measurable sets. This stronger version of the theorem will be used for the model theoretic study of pmp actions, which is the main topic of the present paper.
\\

The formalism of continuous model theory that we use was developed by I.~Ben~Yaacov and A.~Usvyatsov.

While classical model theory is concerned with algebraic theories such as discrete groups, algebraically closed or real closed fields, its continuous counterpart allows the study of metric structures. In recent years, continuous model theory has been used to study theories such as metrics spaces, Banach spaces, Hilbert spaces and measure algebras. More precisely, a particular attention was given to the study of formulas involving automorphisms of the latter theories.
\\

In the present paper we are interested in the model theory of a group action on a probability space, in other words, we look at formulas involving finite subsets of automorphisms of a probability space $(X,\mu)$ from a given subgroup of the group of automorphisms of $(X,\mu)$. However, probability spaces do not admit a model theoretic treatment as such, where the elements of a structure are the points in probability spaces.

In order to solve this issue, we consider as structures not the probability spaces themselves but their associated measure algebra. For a probability space $(X,\Sigma,\mu)$, its associated measure algebra $\mathrm{MAlg}(X,\mu)$ is the quotient set $\Sigma / \pN$ where $\pN$ denotes the $\sigma$-ideal of null sets. It inherits the Boolean operations $\vee, \cap, .^{-1}$ of $\Sigma$ and is endowed with a natural metric $d_\mu(\pi(A),\pi(B)) := \mu(A \bigtriangleup B)$, where $\pi$ is the quotient map.

Moreover, the correspondence between probability spaces and measure algebras is functorial, so that a pmp action on a probability space induces an action by automorphisms on its measure algebra.
\\

Following the latter remarks, we study the model theory of atomless measure algebras with a countable group $\Gamma$ acting by automorphisms. This work follows the one in \cite[Section 18]{yaacovModelTheoryMetric2008} about free actions of $\Z$ and the more general case of free actions of amenable groups treated by A.~Berenstein and C.~W.~Henson in an unpublished paper.

Without loss of generality, we restrict our study to actions of the free group over an infinite countable subset, $F_\infty$, as any action of a countable group can be seen as an action of $F_\infty$. Then one can see that the equivalence relation of elementary equivalence is weaker than approximate conjugacy but stronger than weak equivalence. This result highlights the link between model theory and the equivalence relations usually studied in ergodic theory.

For any IRS $\theta$ on $F_\infty$, we define a theory $\fA_\theta$ axiomatizing pmp actions with IRS $\theta$. By a result of G.~Elek (\cite[Thm. 2]{elekFiniteGraphsAmenability2012}), the hyperfiniteness of an action is determined by its IRS. We thus call an IRS $\theta$ hyperfinite if actions with IRS $\theta$ are hyperfinite.

By Theorem \ref{Elek}, in the context of hyperfinite actions, having the same IRS is equivalent to being elementarily equivalent. We prove:

\begin{thmi}
If $\theta$ is a hyperfinite IRS, then the theory $\fA_\theta$ is complete and model complete.
\end{thmi}

However, unlike in \cite[Section 18]{yaacovModelTheoryMetric2008} these theories do not admit quantifier elimination in general. We nevertherless prove in Theorem \ref{quantifier elimination} that there is a reasonable expansion of the theory which eliminates quantifier, and we then use this to prove

\begin{thmi}
If $\theta$ is a hyperfinite IRS, then the theory $\fA_\theta$ is stable and the stable independence relation given by non dividing admits a natural characterization in terms of the classical probabilistic independence of events (in a sense described in Definition \ref{stability}).
\end{thmi}

\textit{Acknowledgments:}
I am very grateful to my PhD advisors Fran\c{c}ois~Le~Ma\^itre and Todor~Tsankov for suggesting the subject of this paper and for their valuable advice throughout the preparation and writing of this article. I would also like to thank Tom\'as~Ibarluc\'ia and Robin~Tucker-Drob for many helpful discussions and suggestions.


\section{The generalization of Rokhlin Lemma} \label{the rokhlin lemma}

\subsection{Graphings}

\begin{definition}
A \textbf{graph} $G$ is a pair $(\rV(G),\rE(G))$ where $\rV(G)$ is a set and $\rE(G)$ is an irreflexive and symmetric binary relation on $\rV(G)$. Elements of $\rV(G)$ are called \textbf{vertices} of $G$ and elements of $\rE(G)$ are called \textbf{edges} of $G$.
\end{definition}

For $G$ a graph, for each $v \in \rV(G)$ we let $\mathrm{deg}_G(v) = |\{u \in \rV(G) : (v,u) \in \rE(G)\}|$ and we call $\sup_{v \in \rV(G)} \mathrm{deg}_G(v) \in \N \cup \{\infty\}$ the degree bound of $G$.

\begin{definition}
An \textbf{isomorphism between the graphs $G$ and $H$} is a bijection $f \colon \rV(G) \rightarrow \rV(H)$ such that $\forall x,y \in \rV(G)$, $(x,y) \in \rE(G) \Leftrightarrow (f(x),f(y)) \in \rE(H)$.
\end{definition}

\begin{definition}
Let $G$ be a graph, $A \subseteq \rV(G)$ and $B \subseteq \rE(G)$. Then we define :
\begin{itemize}
    \item $\rV_{\inc}^G(B) = \{v \in \rV(G) : \exists u \in \rV(G),\ (u,v) \in B \vee (v,u) \in B\}$ the set of \textbf{vertices incident to $B$}.
    \item $\rE_{\inc}^G(A) = \{(a,v) \in \rE(G) : a \in A\}$ the set of \textbf{edges incident to $A$}.
\end{itemize}

We will write $\rV_{\inc}(B)$ and $\rE_{\inc}(A)$ when the context makes clear which graph $G$ is considered.
\end{definition}

\begin{definition}
Let $G$ be a graph. A \textbf{subgraph} of $G$ is a graph $H$ such that $\rV(H) = \rV(G)$ and $\rE(H) \subseteq \rE(G)$. In this case, we write $H \subseteq G$.

If $\rV \subseteq \rV(G)$, the \textbf{subgraph of $G$ induced by $V$} is the graph $(\rV(G),\rE(G) \cap V \times V)$.  Nevertheless, in many cases it will be convenient to identify the induced graph on $V$ and the graph $(V,\rE(G) \cap V \times V)$ and therefore see the induced graph on $V$ as a graph on the set of vertices $V$.
\end{definition}

In general, we write $G \simeq H$ to indicate that $G$ and $H$ are isomorphic.

\begin{definition}
A \textbf{standard Borel space} is a measurable space isomorphic to $[0,1]$ equipped with its Borel $\sigma$-algebra. We call Borel the maps between two standard Borel spaces which are measurable.
\end{definition}

Let us give some notations regarding probability spaces :
\begin{itemize}
\item If $X$ is a measurable space, we denote by $\fP(X)$ the set of probability measures on $X$.
\item If $(X,\mu)$ is a probability space and $P$ is a property, we write $\forall^* x \in X\ P(x)$ for $\mu(\{x \in X : P(x)\})=1$ and $\exists^* x \in X\ P(x)$ for $\mu(\{x \in X : P(x)\})>0$.
\item If $(X,\mu)$ is a probability space, $Y$ is a measurable space and $T : X \rightarrow Y$ is a measurable map, we write $T_*\mu$ for the pushforward of $\mu$ by $T$, that is the measure in $\fP(Y)$ defined by $T_*\mu(A) = \mu(T^{-1}(A))$ for any Borel subset $A \subseteq Y$.
\end{itemize}

\begin{definition}
Let $X$ be a standard Borel space and $\mR$ be a Borel (as a subset of the measurable space $X \times X$) equivalence relation on $X$. We let $[\mR]$ be the group of Borel automorphisms of $X$ whose graphs are contained in $\mR$. We say that a Borel probability measure $\mu$ on $X$ is \textbf{$\mR$-invariant} if every element of $[\mR]$ preserves the measure $\mu$, namely, $\forall T \in [\mR],\ T_*\mu = \mu$.
\end{definition}

\begin{proposition}[{\cite[Section 8]{kechrisIIAmenabilityHyperfiniteness2004}}] \label{left and right measures}
With the same notations as above, for any  $\mu \in \fP(X)$, we can define two measures $\mu_l$ and $\mu_r$ on $\mR$ by
\begin{itemize}
    \item for all non-negative Borel $f \colon \mR \rightarrow [0,\infty]$, $\int_\mR f\ d\mu_l = \int_X \underset{y \in [x]_\mR}{\sum} f(x,y)\ d\mu(x),$
    \item for all non-negative Borel $f \colon \mR \rightarrow [0,\infty]$, $\int_\mR f\ d\mu_r = \int_X \underset{y \in [x]_\mR}{\sum} f(y,x)\ d\mu(x),$
\end{itemize}
where $[x]_\mR$ denotes the equivalence class of $x$ for $\mR$.
Then $\mu_l = \mu_r$ if and only if $\mu$ is $\mR$-invariant.
\end{proposition}

\begin{definition}
Let $\pG$ be a Borel graph on a standard probability space $(X,\mu)$ which has countable connected components. Then the  equivalence relation $\mR_\pG$ induced by $\pG$ is the equivalence relation on $(X,\mu)$ whose classes are the connected components of $\pG$. By the Lusin-Novikov theorem, $\mR_\pG$ is a Borel equivalence relation. We say that $\pG$ is a \textbf{graphing} when $\mu$ is $\mR_\pG$-invariant.
\end{definition}

We can define a measure on the set of edges of a graphing by:

\begin{definition}
Let $\pG(X,\mu)$ be a graphing and $Z \subseteq \rE(\pG)$ be a Borel set. The \textbf{edge measure} of the set $Z$ is defined by $\mu_\rE(Z) := \mu_l(Z) = \mu_r(Z)$, where $\mu_l$ and $\mu_r$ are defined with respect to the Borel equivalence relation $\mR_\pG$.
\end{definition}

For a graphing of degree bound $d$, the edge measure of a set of edges is bounded by the measure of the vertices incident to this set. Namely, for all Borel $Z \subseteq \rE(\pG)$ we have
$$\frac{1}{2}\mu(\rV_{\inc}(Z)) \leq \mu_\rE(Z) \leq d \mu(\rV_{\inc}(Z)).$$

\subsection{Classical Rokhlin Lemma}

A measure-preserving transformation is called \textbf{aperiodic} if almost all its orbits are infinite.

Rokhlin Lemma states that if $T$ is an aperiodic measure-preserving transformation of a  standard probability space $(X,\mu)$, then for every $n \in \N$ and every $\varepsilon >0$, there is a Borel subset $A \subseteq X$ such that the sets $A,T A,\dots, T^{n-1} A$ are pairwise disjoint and
$$\mu \left( \overset{n-1}{\underset{i=0}{\bigsqcup}} T^i A \right) > 1 - \varepsilon .$$

What we present in this paper is not a generalization of Rokhlin Lemma itself but rather of one of its important and well-known consequences:

\begin{corollary}[Uniform Approximation Theorem, {\cite[Theorem 2.2]{kechrisGlobalAspectsErgodic2010}}] \label{rokhlin corollary}
Any two aperiodic measure-preserving transformations $\tau_1$ and $\tau_2$ on standard probability spaces $(X,\mu)$ and $(Y,\nu)$ are approximately conjugate.
\end{corollary}

An aperiodic measure-preserving transformation can be seen as a free action of $\Z$. The goal of this section is to generalize the latter Corollary to hyperfinite actions of a countable group which have a given IRS (i.e. Invariant Random Subgroup, defined in subsection \ref{irs}).

\subsection{Hyperfiniteness} \label{hyperfiniteness}

The key point on the proof of Uniform Approximation Theorem \ref{rokhlin corollary} is that the dynamics of an aperiodic automorphism are understood on arbitrary large sets. In the section we define the notion of hyperfiniteness of a pmp action, which allows one to make this idea work in a much more general context.
\\

\begin{definition}[See "approximately finite group" in \cite{Dye1959}]
A pmp action $\Gamma \act{} (X,\mu)$ is said to be \textbf{hyperfinite} if for every finite $S \subseteq \Gamma$ and every $\varepsilon>0$, there exists a finite group $G$ acting in a measure-preserving way on $(X,\mu)$ such that
$$\mu \left( \{ x \in X : S \cdot x \subseteq G \cdot x \} \right) > 1 - \varepsilon.$$
\end{definition}

What we are mostly interested in is the characterization of hyperfiniteness for graphings.

\begin{definition}
Let $\pG(X,\mu)$ be a graphing. $\pG$ is called \textbf{hyperfinite} if for any $\varepsilon >0$ there exists $M \in \N$ and a Borel set $Z \subseteq \rE(\pG)$ such that $\mu_\rE(Z) < \varepsilon$ and the subgraphing $\pH = \pG \setminus Z$ has components of size at most $M$.
\end{definition}

\begin{definition}
Let $F$ be a finite set. An \textbf{$F$-colored graphing} on a standard probability space $(X,\mu)$ is a graphing $\pG(X,\mu)$ endowed with a Borel map $\varphi_\pG \colon \rE(\pG) \rightarrow F$. For $(x,y) \in \rE(\pG)$, we call $\varphi_\pG(x,y)$ the color of $(x,y)$.

Additionally, for $c \in F$, we write $\rE^c(\pG)$ for the set of edges colored by $c$, namely $\varphi_\pG^{-1}(c)$.
\end{definition}

We will simply write $\pG$ and consider the color implicitly when dealing with colored graphings.

\begin{definition}
Let $\pG(X,\mu)$ and $\pG'(Y,\nu)$ be two $F$-colored graphings. A \textbf{colored graphing factor map} $\pi \colon Y \rightarrow X$ is a pmp map such that for almost all $y \in Y$, $\pi \restriction_{[y]_{\pG'}}$ is an isomorphism of $F$- colored graphs.

We say that $\pG$ is a colored factor of $\pG'$ and we write $\pG \underset{c}{\sqsubseteq} \pG'$ if there is a colored factor map $\pi \colon Y \rightarrow X$.
\end{definition}

Let $\Gamma$ be a group and $S$ be a finite subset of $\Gamma$. Let us consider a measure-preserving action $\Gamma \act{\alpha} (X,\mu)$. We define a $\mP(S)$-colored graphing $\pG_{\alpha,S}$ on $(X,\mu)$ by $(x,y) \in \rE(\pG_{\alpha,S})$ if and only if there is a $s \in S$ such that $y=sx$ and we color the edges of $\pG_{\alpha,S}$ by letting the color of an edge $(x,y)$ be $\{s \in S : y=sx\}$. We call it the \textbf{Schreier graph} of the action $\alpha$ relative to $S$.

\begin{lemma} \label{hyperfiniteness passes to graphings}
Let $\Gamma$ be a countable group and let $\Gamma \act{\alpha} (X,\mu)$ be a pmp action. Then $\alpha$ is hyperfinite if and only if for every finite $S \subseteq \Gamma$, $\pG_{\alpha,S}$ is hyperfinite.
\end{lemma}

\begin{proof}
Suppose $\alpha$ is hyperfinite and let $S \subseteq \Gamma$ be finite and $\varepsilon >0$.

By hyperfiniteness, there exists a finite group $G$ along with a pmp action $G \act{} (X,\mu)$ such that $\mu(\{x \in X : S \cdot x \subseteq G \cdot x\}) > 1 - \varepsilon$. In particular, when restricted to the set $\{x \in X : S \cdot x \subseteq G \cdot x\}$, the Schreier graph $\pG_{\alpha,S}$ has finite components of size less than $|G|$.
\\

For the converse, suppose that for any $S \subseteq \Gamma$ finite, the graphing $\pG_{\alpha,S}$ is hyperfinite.

Let $S \subseteq \Gamma$ be finite and let $\varepsilon >0$. Then there exist $Z \subseteq \rE(\pG_{\alpha,S})$ Borel and $M \in \N$ such that $\mu_\rE(Z) < \frac{\varepsilon}{2}$ and $\pG_{\alpha,S} \setminus Z$ has components of size at most $M$.

We define a pmp action of $\prod_{n \leq M} \Z / n\Z$ on $(X,\mu)$ as follows :

Since $(X,\mu)$ is a standard probability space, there is a Borel linear ordering $<$ of $X$. This induces, for $n \leq M$, an action of $\Z / n\Z$ on the set of elements of $\pG_{\alpha,S} \setminus Z$ whose component is of size $n$ by shifting any component according to the order $<$.

It follows that $\prod_{n \leq M} \Z / n \Z$ acts as a product on $X \setminus Z$ in a pmp way, and we extend this action to the whole $X$ by letting $\prod_{n \leq M} \Z / n \Z$ act trivially on $Z$.

One can easily check that for $x \notin \rV_{\inc}(Z)$, $S \cdot x$ is exactly the set of neighbors of $x$ in $\pG_{\alpha,S} \setminus Z$ and thus it is contained in $\left[ x \right]_{\pG_{\alpha,S} \setminus Z} = \left( \prod_{n \leq M} \Z / n \Z \right) \cdot x$. Moreover, $\mu(\rV_{\inc}(Z)) \leq 2 \mu_\rE(Z) < \varepsilon$ so we conclude that $\alpha$ is hyperfinite.
\end{proof}

\subsection{Invariant Random Subgroups} \label{irs}

Let $\Gamma \act{\alpha} (X,\mu)$ be a measure-preserving action of the countable group $\Gamma$. With this action we can associate a probability measure on the Polish space of subgroups of $\Gamma$ as follows. Consider the compact Polish space $\{0,1\}^\Gamma$. We let $\mathrm{Sub}(\Gamma)$ be the closed subset of $\{0,1\}^\Gamma$ consisting of the subgroups of $\Gamma$. Then $\mathrm{Sub}(\Gamma)$ is a compact Polish space.

We have a natural map $\mathrm{Stab}^\alpha \colon X \rightarrow \mathrm{Sub}(\Gamma)$ defined by $x \mapsto \mathrm{Stab}^\alpha(x) = \{g \in \Gamma : g^\alpha (x) = x\}$ and that gives us a probability measure $\mathrm{Stab}^\alpha_*\mu \in \fP(\mathrm{Sub}(\Gamma))$ that we call the Invariant Random Subgroup (IRS in short) of $\alpha$ and denote by $\theta_\alpha$. Moreover, $\Gamma$ acts on $\mathrm{Sub}(\Gamma)$ by conjugacy and the well known formula $\mathrm{Stab}^\alpha(gx) = g \mathrm{Stab}^\alpha(x) g^{-1}$ implies that the map $\mathrm{Stab}^\alpha$ is equivariant. Therefore, $\theta_\alpha$ is a $\Gamma$-invariant measure on $\mathrm{Sub}(\Gamma)$. We thus define the general notion of an \textbf{IRS} on $\Gamma$ to be a probability measure on $\mathrm{Sub}(\Gamma)$ invariant for the action $\Gamma \act{} \mathrm{Sub}(\Gamma)$ by conjugacy.
\\

G.~Elek proved in \cite[Thm. 2]{elekFiniteGraphsAmenability2012} that two pmp actions of a countable group $\Gamma$ with the same IRS are either both hyperfinite or both non-hyperfinite.
\\

Moreover, Abert, Glasner and Virag proved in \cite[Prop. 13]{abertKestenTheoremInvariant2014} that any IRS can be obtained as the IRS associated to a pmp action.

We can thus express hyperfiniteness as a property of the IRS itself:

\begin{definition}
Let $\Gamma$ be a countable group. An IRS $\theta$ on $\Gamma$ is called \textbf{hyperfinite} if one of the following two equivalent statements is satisfied :
\begin{enumerate}
    \item There exists a hyperfinite pmp action which has IRS $\theta$.
    \item Every pmp action which has IRS $\theta$ is hyperfinite.
\end{enumerate}
\end{definition}

\begin{definition} \label{factor}
Let $\Gamma \act{\alpha} (X,\mu)$ and $\Gamma \act{beta} (Y,\nu)$. An \textbf{action factor map} $\pi \colon Y \rightarrow X$ is a measure-preserving map such that $\forall^* y \in Y\ \forall \gamma \in \Gamma,\ \pi(\gamma^\beta y) = \gamma^\alpha \pi(y)$.

We say that $\alpha$ \textbf{is a factor of} $\beta$ and we write $\alpha \sqsubseteq \beta$ if there exists an action factor map $\pi \colon Y \rightarrow X$.
\end{definition}

\begin{lemma} \label{action factor with same irs implies same stabilizers}
Let $\alpha,\beta$ be two actions of a countable group $\Gamma$ on standard probability spaces $(X,\mu)$ and $(Y,\nu)$. Suppose that there is an action factor map $\pi \colon Y \rightarrow X$ for $\alpha$ and $\beta$ and that $\theta_\alpha = \theta_\beta$. Then $\forall^* y \in Y,\ \mathrm{Stab}^\alpha(\pi(y)) = \mathrm{Stab}^\beta(y)$.
\end{lemma}

\begin{proof}
For $\gamma \in \Gamma$, let $\rN_\gamma = \{\Lambda \in \mathrm{Sub}(\Gamma) : \gamma \in \Lambda\}$. Then $(\rN_\gamma)_{\gamma \in \Gamma}$ is a subbasis of the topology of $\mathrm{Sub}(\Gamma)$ consisting of clopen sets and any measure on $\mathrm{Sub}(\Gamma)$ is determined by the values it takes on this subbasis.

By the definition of action factor map, we have $\forall^*y\ \mathrm{Stab}^\beta(y) \subseteq \mathrm{Stab}^\alpha(\pi(y))$. Suppose now that $\exists^*y\ \mathrm{Stab}^\beta(y) \subsetneq \mathrm{Stab}^\alpha(\pi(y))$.

By countability of $\Gamma$, $\exists \gamma \in \Gamma\ \exists^*y$, $\gamma \in \mathrm{Stab}^\alpha(\pi(y)) \setminus \mathrm{Stab}^\beta(y)$, thus
\begin{eqnarray*}
\theta_\beta(\rN_\gamma) & = & \mathrm{Stab}^\beta_*\nu(\rN_\gamma) \\
                       & < & (\mathrm{Stab}^\alpha \circ \pi)_*\nu(\rN_\gamma)\\
                       & = & \mathrm{Stab}^\alpha_*(\pi_*\nu)(\rN_\gamma) \\
                       & = & \mathrm{Stab}^\alpha_*\mu(\rN_\gamma) \\
                       & = & \theta_\alpha(\rN_\gamma),
\end{eqnarray*}
a contradiction.
\end{proof}

\begin{corollary} \label{action factor with same irs implies graphing factor}
Let $\alpha,\beta$ be actions of a countable group $\Gamma$ on standard probability spaces $(X,\mu)$ and $(Y,\nu)$ such that $\alpha \sqsubseteq \beta$ and $\theta_\alpha = \theta_\beta$, and let $S \subseteq \Gamma$ be finite . Then we have $\pG_{\alpha,S} \underset{c}{\sqsubseteq} \pG_{\beta,S}$ as $\mP(S)$-colored graphings.
\end{corollary}

\begin{proof}
Applying Lemma \ref{action factor with same irs implies same stabilizers} to an action factor map $\pi \colon Y \rightarrow X$ gives us that for almost every $y \in Y$, $\pi \restriction_{ \Gamma \cdot y}$ is a $\Gamma$-equivariant bijection $\Gamma \cdot y \rightarrow \Gamma \cdot \pi(y)$ and so it is an isomorphism of Schreier graphs. It follows that $\pi$ is a graphing factor map.
\end{proof}

\subsection{The proof of Theorem A}

\subsubsection{The preliminary case of factors}

We begin with the case where one of the actions is a factor of the other. In fact we prove a stronger version involving the stability of Borel sets.

\begin{definition}
Let $F_1,F_2$ be two finite sets. An \textbf{$(F_1,F_2)$-bicolored graphing} on a standard probability space $(X,\mu)$ is a graphing $\pG(X,\mu)$ endowed with two Borel maps $\varphi_\pG \colon \rE(\pG) \rightarrow F_1$ and $\psi_\pG \colon X \rightarrow F_2$. We call $\psi_\pG(x)$ the vertex-color of $x$ and $\varphi_\pG(x,y)$ the edge-color of $(x,y)$.
\end{definition}

\begin{definition}
Let $\pG(X,\mu)$ and $\pG'(Y,\nu)$ be two $(F_1,F_2)$-bicolored graphings. A \textbf{bicolored graphing factor map} $\pi \colon Y \rightarrow X$ is an $F_1$-colored graphing factor map such that $\psi_\pG \circ \pi = \psi_{\pG'}$.

We say that $\pG$ is a bicolored factor of $\pG'$ and we write $\pG \underset{bic}{\sqsubseteq} \pG'$ if there is a bicolored factor map $\pi \colon Y \rightarrow X$.
\end{definition}

\begin{theorem}[Approximate parametrized conjugacy for factor actions] \label{approximate parametrized conjugacy for factor actions}
Let $(X,\mu)$ and $(Y,\nu)$ be standard probability spaces and $A_1,\dots,A_k \subseteq X$, $B_1,\dots,B_k \subseteq Y$ be Borel subsets. Let $\Gamma$ be a countable group, $\theta$ be a hyperfinite IRS on $\Gamma$ and $\Gamma \act{\alpha} (X,\mu)$,$\Gamma \act{\beta} (Y,\nu)$ be pmp actions of $\Gamma$ with IRS $\theta$ and such that $\alpha \sqsubseteq \beta$ for an action factor map $\pi \colon Y \rightarrow X$ such that $\forall i \leq k,\ \pi^{-1}(A_i) = B_i$. Then for $\varepsilon >0$ and $\gamma_1,\dots,\gamma_n \in \Gamma$, there exists a pmp bijection $\rho \colon X \rightarrow Y$ such that $\forall i \leq k,\ \rho(A_i) = B_i$ and
$$\mu(\{x \in X : \forall i \leq n,\ \rho \circ \gamma_i^\alpha(x) = \gamma_i^\beta \circ \rho(x)\}) > 1 - \varepsilon.$$
\end{theorem}

\begin{proof}
We begin the proof with a claim about graphings.
\begin{pushright}
\begin{claim} \label{approximate colored conjugacy for factor graphings}
Let $\pG(X,\mu)$ and $\pG'(Y,\nu)$ be hyperfinite $(F_1,F_2)$-bicolored graphings of degree bound at most $d$ such that $\pG(X,\mu) \underset{bic}{\sqsubseteq} \pG'(Y,\nu)$. Then for any $\varepsilon >0$ there exists a pmp bijection $\rho \colon X \rightarrow Y$ such that $\psi_\pG = \psi_{\pG'} \circ \rho$ and
$$\mu_\rE \left( \underset{c \in F_1}{\bigcup} \rho^{-1} \left( \rE^c(\pG') \right) \bigtriangleup \rE^c(\pG) \right) < \varepsilon.$$
\end{claim}

\begin{claimproof}
Let $\pi$ be a bicolored graphing factor map $Y \rightarrow X$.
First take a Borel set $Z \subseteq \rE(\pG)$ of measure less than $\frac{\varepsilon}{2d}$ and $M \in \N$ such that the graphing $\pH = \pG \setminus Z$ has components of size at most $M$.
Let $Z'=\pi^{-1}(Z)$ and $\pH' = \pG' \setminus Z'$. Since $\pi$ is a graphing factor map, we know that $\pH'$ has components of size at most $M$. Then $\pH$ and $\pH'$ have a $(F_1,F_2)$-bicolored graphing structure respectively for the maps $\varphi_\pG \restriction _{\rE(\pH)}$, $\psi_\pG$ and $\varphi_{\pG'} \restriction _{\rE(\pH')}$, $\psi_{\pG'}$.

Consider the set $\mG_M$ of connected $(F_1,F_2)$-colored graphs of size at most $M$. We consider the two partitions $X = \underset{S \in \mG_M}{\bigsqcup} C_S^\pH$ and $Y = \underset{S \in \mG_M}{\bigsqcup} C_S^{\pH'}$, where $C_S^\pH$ is defined to be the set of vertices of $\pH$ whose component is $(F_1,F_2)$-colored isomorphic to $S$. Since $\pi$ induces $(F_1,F_2)$- colored graph isomorphisms, we have $C_S^{\pH'}=\pi^{-1}(C_S^\pH)$.

In order to define $\rho$, it suffices to define a measure-preserving bijection $\rho_S \colon C_S^\pH \rightarrow C_S^{\pH'}$ preserving bicolored graph structures for each $S \in \mG_M$.

Indeed, the union of all these bijections would yield a measure-preserving bijection  $\rho \colon X \rightarrow Y$ preserving vertex-colors such that  $\forall x \in X \setminus \rV_{\inc}(Z),\ \B^\pG(x,1) = \B^\pH(x,1) \simeq \B^{\pH'}(\rho(x),1) = \B^{\pG'}(\rho(x),1)$, where $\B^G(v,n)$ denotes the ball of size $n$ centered at $v$ in the graph $G$. Hence we would have $\rV_{\inc} \left( \underset{c \in F_1}{\bigcup} \rho^{-1} \left( \rE^c(\pG') \right) \bigtriangleup \rE^c(\pG) \right) \subseteq \rV_{\inc}(Z)$, and so
$$\mu_\rE \left( \underset{c \in F_1}{\bigcup} \rho^{-1} \left( \rE^c(\pG') \right) \bigtriangleup \rE^c(\pG) \right) \leq d \mu(\rV_{\inc}(Z)) \leq 2d \mu_\rE(Z) < \varepsilon.$$

Take $S \in \mG_M$ and let us define $\rho_S$. First we define a partition of $C_S^\pH$ into Borel transversals $(T_v)_{v \in \rV(S)}$ (for $\pH$) by induction, such that the elements of $T_v$ occupy the same place in their component for $\pH$ as $v$ in $S$.

Suppose that the $T_{v'}$ are already defined for $v' \in R$ where $R$ is a proper subset of $\rV(S)$. Take $v \in \rV(S) \setminus R$ incident to $R$ and let $\widetilde{T_v} = \{x \in C_S^\pH : ([x]_\pH,x) \simeq_R (S,v)\}$. Here $\simeq_R$ means isomorphic over $R$, that is there exists an isomorphism $f \colon ([x]_\pH,x) \rightarrow (S,v)$ of colored rooted graphs such that $\forall v' \in R,\ f([x]_\pH \cap T_{v'}) = \{v'\}$. Now since $\pH$ has finite components, chose for $T_v$ any Borel transversal of $\widetilde{T_v}$. Then we let $R' = R \cup \{v\}$ and we iterate the construction.

Again since $\pi$ is a bicolored graphing factor map, the family $(\pi^{-1}(T_v))_{v \in \rV(S)}$ is a partition of $C_S^{\pH'}$ into Borel transversals (for $\pH'$) such that the elements of $\pi^{-1}(T_v)$ occupy the same place in their component for $\pH'$ as $v$ in $S$. We may now define $\rho_S$:
\begin{itemize}
    \item We start by chosing $v_0 \in S$ and taking a measure-preserving bijection $\rho_S^{v_0} \colon T_{v_0} \rightarrow \pi^{-1}(T_{v_0})$.
    \item Then for every $v \in S$, there is a unique way of extending $\rho_S^{v_0}$ to $T_v$ while respecting the graph structure of $S$. Indeed, take $x \in T_v$, there is a unique $x_0 \in [x]_\pH \cap T_{v_0}$ and we want to define $\rho_S^v(x) \in [\rho_S^{v_0}(x_0)]_{\pH'} \cap \pi^{-1}(T_v)$ but again this intersection is a singleton. Define $\rho_S \colon C_S^\pH \rightarrow C_S^{\pH'}$ to be this unique extension of $\rho_S^{v_0}$ satisfying the condition above.
    
    As $\pi$ is a colored graphing factor map, it is clear that $\rho_S$ is a measure-preserving bijection and that for every $x \in C_S^\pH$, $\rho_S$ induces an isomorphism of colored graphs between $[x]_\pH$ and $[\rho_S(x)]_{\pH'}$.
\end{itemize}
\end{claimproof}
\end{pushright}
We now want to apply the Claim to suitable graphings to conclude. Let  $S = \{\gamma_1,\dots,\gamma_n,\gamma_1^{-1},\dots,\gamma_n^{-1}\}$ and consider the graphings $\pG_{\alpha,S}$ and $\pG_{\beta,S}$.

For the spaces of colors, we choose $F_1 = \mP(S)$ and $F_2 = \mP(\{1,\dots,k\})$. The way we color edges has already been explained; for vertices, simply color a vertex $x \in X$ by  $\psi_{\pG_{\alpha,S}}(x) = \{i \leq k : x \in A_i\}$ and $y \in Y$ by $\psi_{\pG_{\beta,S}}(y) = \{i \leq k : y \in B_i\}$.

First, $\pG_{\alpha,S}$ and $\pG_{\beta,S}$ are indeed $(\mP(S),\mP(\{1,\dots,k\}))$-bicolored graphings, and are hyperfinite since $\alpha$ and $\beta$ are hyperfinite actions.

The next step is to prove that $\pi$ considered in the statement of the theorem is a bicolored factor map for the $(\mP(S),\mP(\{1,\dots,k\}))$-bicolored graphings $\pG_{\alpha,S}$ and $\pG_{\beta,S}$.

\begin{itemize}
    \item First, $\pi$ is indeed a pmp map $Y \rightarrow X$.
    \item Then for $y \in Y$, we have $$\psi_{\pG_{\alpha,S}}(\pi(y)) = \{i \leq k : \pi(y) \in A_i\} = \{i \leq k : y \in B_i\} = \psi_{\pG_{\beta,S}}(y).$$
    \item Finally, by Corollary \ref{action factor with same irs implies graphing factor}, $\pi$ is furthermore a colored graphing factor map between the  $\mP(S)$-colored graphings $\pG_{\alpha,S}$ and $\pG_ {\beta,S}$.
\end{itemize}

Applying the Claim gives us a pmp bijection $\rho \colon X \rightarrow Y$ such that $\psi_{\pG_{\alpha,S}} = \psi_{\pG_{\beta,S}} \circ \rho$ and
$$\mu_\rE \left( \underset{c \in \mP(S)}{\bigcup} \rE^c(\pG_{\alpha,S}) \bigtriangleup \rho^{-1} \left( \rE^c(\pG_{\beta,S}) \right) \right) < \frac{\varepsilon}{2}.$$
But then for $1 \leq i \leq k$, $\rho(A_i) = B_i$, and by definitions of $\pG_{\alpha,S}$ and $\pG_{\beta,S}$ we get
$$\{x \in X : \exists \gamma \in S,\ \rho \circ \gamma^\alpha(x) \neq \gamma^\beta \circ \rho(x)\} \subseteq \rV_{\inc} \left( \underset{c \in \mP(S)}{\bigcup} \rE^c(\pG_{\alpha,S}) \bigtriangleup \rho^{-1} \left( \rE^c(\pG_{\beta,S}) \right) \right),$$
so its measure is less than $2 \cdot \frac{\varepsilon}{2} = \varepsilon$.
\end{proof}

\subsubsection{Amalgamation of measure-preserving actions} \label{amalgamation of actions}

To conclude the proof of Theorem \ref{approximate conjugacy for pmp actions with a given hyperfinite irs}, we will use the transitivity of the approximate conjugacy relation and show that for any two pmp actions $\Gamma \act{\alpha} (X,\mu)$ and $\Gamma \act{\beta} (Y,\nu)$ of $\Gamma$ such that $\theta_\alpha = \theta_\beta$, there is a third pmp action $\Gamma \act{\zeta} (Z,\eta)$ of IRS $\theta$ such that both $\alpha$ and $\beta$ are factors of $\zeta$.
\\

We recall the definition of the relative independent joining following the presentation in \cite{ErgodicTheoryViaJoinings}.

\begin{proposition}[Disintegration theorem,{\cite[A.7]{ErgodicTheoryViaJoinings}}] \label{disintegration}
Let $X,Y$ be standard probability spaces, $\mu \in \fP(Y)$ and $\pi \colon Y \rightarrow X$ be a measurable map. We let $\nu = \pi_* \mu$. Then there is a $\nu$-a.e. uniquely determined family of probability measures $(\mu_x)_{x \in X} \in \fP(Y)^X$ such that:
\begin{enumerate}
    \item For each Borel $B \subseteq Y$, the map $x \mapsto \mu_x(B)$ is measurable.
    \item For $\nu$-a.e. $x \in X$, $\mu_x$ is concentrated on the fiber $\pi^{-1}(x)$.
    \item For every Borel map $f \colon Y \rightarrow [0,\infty]$, $\int_Y f(y)\ d\mu(y) = \int_X \int_Y f(y)\ d\mu_x(y)\ d\nu(x)$.
\end{enumerate}
We then write $\mu = \int_X \mu_x\ d\nu$.
\end{proposition}

\begin{definition}[{\cite[Section 6.1]{ErgodicTheoryViaJoinings}}] \label{relative independent joining}
Let $\Gamma \act{\alpha} (X,\mu)$ and $\Gamma \act{\beta} (X',\mu')$ be pmp actions on standard probability spaces, and let $\Gamma \act{\xi} (Y,\nu)$ be an action on a standard probability space common factor of $\alpha$ and $\beta$ for respective action factor maps $\pi \colon X \rightarrow Y$ and $\pi' \colon X' \rightarrow Y$.

We can disintegrate $\mu$ and $\mu'$ with respect to $\nu$ using the Borel maps $\pi$ and $\pi'$ to get $\mu = \int_Y \mu_y\ d\nu$ and $\mu' = \int_Y \mu'_y\ d\nu$.

Consider $Z := X \times Y$ and $\eta \in \fP(Z)$ defined by $\eta = \int_Y \mu_y \times \mu'_y\ d\nu$.

The pmp action $\Gamma \act{\alpha \times \beta} (Z,\eta)$ is called the \textbf{independent joining of $\alpha$ and $\beta$ over $\xi$} and is denoted by $\alpha \underset{\xi}{\times} \beta$.
\end{definition}

The action $\alpha \underset{\xi}{\times} \beta$ is indeed a \textbf{joining} of $\alpha$ and $\beta$ over $\xi$, meaning that both $\alpha$ and $\beta$ are factors of their independent joining over $\xi$, respectively for the projections on the first and second coordinates $p_1$ and $p_2$, and moreover the following diagram commutes, up to a null set:

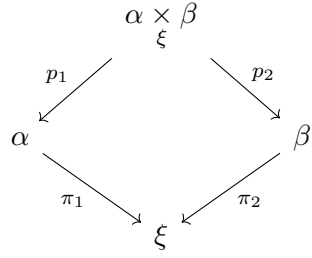
\begin{center}
\begin{tikzcd}
& \ar[ld, "p_1"'] \alpha \underset{\xi}{\times} \beta \ar[rd, "p_2"] & \\
\alpha \ar[rd, "\pi_1"'] & & \ar[ld, "\pi_2"] \beta \\
& \xi &
\end{tikzcd}
\end{center}

Let $\theta$ be an IRS on $\Gamma$, we write $\boldsymbol{\theta}$ for the measure-preserving conjugation action $\Gamma \act{\boldsymbol{\theta}} (\mathrm{Sub}(\Gamma),\theta)$. For every pmp action $\Gamma \act{\alpha} (X,\mu)$, the map $\mathrm{Stab}^\alpha \colon (X,\mu) \rightarrow (\mathrm{Sub}(\Gamma),\theta)$ is an action factor map.

\begin{lemma} \label{joining over an irs}
Let $\Gamma$ be a countable group and $\theta$ be an IRS on $\Gamma$. Let $\Gamma \act{\alpha} (X,\mu)$, $\Gamma \act{\beta} (Y,\nu)$ be pmp actions of IRS $\theta$. Then $\alpha \underset{\boldsymbol{\theta}}{\times} \beta$ has IRS $\theta$.
\end{lemma}

\begin{proof}
Let $\zeta$ denote $\alpha \underset{\boldsymbol{\theta}}{\times} \beta$. We know that the following diagram commutes.
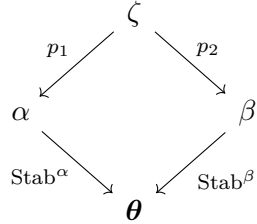
\begin{center}
\begin{tikzcd}
& \ar[ld, "p_1"'] \zeta \ar[rd, "p_2"] & \\
\alpha \ar[rd, "\mathrm{Stab}^\alpha"'] & & \ar[ld, "\mathrm{Stab}^\beta"] \beta \\
&\boldsymbol{\theta}&
\end{tikzcd}
\end{center}

Therefore, for $\gamma \in \Gamma$, we have  $\forall^*(x,y),\ \gamma x = x \Leftrightarrow \gamma y = y \Leftrightarrow \gamma (x,y) = (x,y)$. It follows that  $\forall^*(x,y),\ \mathrm{Stab}^\zeta(x,y) = \mathrm{Stab}^\alpha(x)$ or in other words, $\mathrm{Stab}^\zeta = \mathrm{Stab}^\alpha \circ p_1$. We conclude that
$$\theta_\zeta = \mathrm{Stab}^\zeta_*\eta = \mathrm{Stab}^\alpha_* \left( {p_1}_*\eta \right) = \mathrm{Stab}^\alpha_*\mu = \theta_\alpha = \theta.$$
\end{proof}

Theorem \ref{approximate conjugacy for pmp actions with a given hyperfinite irs} states that if $\alpha$ and $\beta$ are two pmp hyperfinite actions of a group $\Gamma$ on a standard probability space such that $\theta_\alpha = \theta_\beta$, then $\alpha$ and $\beta$ are approximately conjugate. We can now prove this theorem:

\begin{proof}
Let $\Gamma \act{\alpha} (X,\mu)$ and $\Gamma \act{\beta} (Y,\nu)$  be two hyperfinite actions of $\Gamma$ having IRS $\theta$ and consider the joining $\Gamma \act{\zeta} (Z,\eta)$ from Lemma \ref{joining over an irs}.

Applying twice Theorem \ref{approximate parametrized conjugacy for factor actions} with no Borel parameters we get two pmp bijections $\rho \colon X \rightarrow Z$ and $\rho' \colon Y \rightarrow Z$ such that:
$$\mu(\{x \in X : \forall i \leq n,\ \rho \circ \gamma_i^\alpha(x) = \gamma_i^\zeta \circ \rho(x)\}) > 1 - \frac{\varepsilon}{2}$$
and
$$\nu(\{y \in Y : \forall i \leq n,\ \rho' \circ \gamma_i^\beta(y) = \gamma_i^\zeta \circ \rho'(y)\}) > 1 - \frac{\varepsilon}{2}.$$
Thus, $\rho'^{-1} \circ \rho \colon X \rightarrow Y$ witnesses the $\varepsilon$-approximate conjugacy of $\alpha$ and $\beta$.
\end{proof}


\section{Model theory of hyperfinite actions} \label{model theory of hyperfinite actions}

\subsection{Measure algebras}

The reader unfamiliar with continuous model theory is referred to \cite{yaacovModelTheoryMetric2008}. We will use the same notations as theirs.

\begin{definition}
A \textbf{measure algebra} is a Boolean algebra $(\pA,\vee,\wedge,\neg,0,1,\subseteq,\bigtriangleup)$ endowed with a function $\mu \colon \pA \rightarrow [0,1]$ satisfying the following :
\begin{enumerate}
    \item $\mu(1) = 1$.
    \item $\forall a,b \in \pA$, $\mu(a \wedge b) =0 \Rightarrow \mu(a \vee b) = \mu(a) + \mu(b)$.
    \item The function $d_\mu(a,b) \coloneqq \mu(a \bigtriangleup b)$ is a complete metric on $\pA$.
\end{enumerate}
\end{definition}

\begin{proposition}[{\cite[323G c)]{fremlinMeasureTheoryVol2002}}]
Any measure algebra $\pA$ is Dedekind complete, meaning that any subset $S \subseteq \pA$ admits a supremum and an infimum, that we respectively denote by $\bigvee S$ and $\bigwedge S$.
\end{proposition}

\begin{definition}
An element $a \in \pA$ is an \textbf{atom} if $\forall b \in \pA,\ b \subseteq a \Rightarrow b \in \{0,a\}$.
A measure algebra is \textbf{atomless} if it has no atoms.
\end{definition}

\begin{proposition}[{\cite[331C]{fremlinMeasureTheoryVol2002}}]
If a measure algebra $\pA$ is atomless, then
$$\forall a \in \pA\ \forall r \in [0,\mu(a)]\ \exists b \subseteq a,\ \mu(b) =r.$$
\end{proposition}

We introduce the classical example of a measure algebra: For $(X,\mu)$ a probability space, we let $\mathrm{MAlg}(X,\mu)$ be the quotient of the Boolean algebra of measurable subsets of $X$ by the $\sigma$-ideal of null sets. For $A \subseteq X$ Borel we denote its class in $\mathrm{MAlg}(X,\mu)$ by $\left[ A \right]_\mu$. The measure $\mu$ descends to the quotient $\mathrm{MAlg}(X,\mu)$ and then $\mathrm{MAlg}(X,\mu)$ endowed with $\mu$ is a measure algebra. When $(X,\mu)$ is a standard probability space, $\mathrm{MAlg}(X,\mu)$ is atomless and separable for the topology induced by $d_\mu$.
\\

Conversely, we have:

\begin{proposition} [{\cite[331L]{fremlinMeasureTheoryVol2002}}] \label{measure algebras arise from probability spaces}
Let $\pA$ be a separable atomless measure algebra. Then there exists a standard probability space $(X,\mu)$ such that $\pA$ is isomorphic to $\mathrm{MAlg}(X,\mu)$.
\end{proposition}

Let $f \colon (X,\mu) \rightarrow (Y,\nu)$ be a measure-preserving map. Then the map $\widetilde{f} \colon \mathrm{MAlg}(Y,\nu) \rightarrow \mathrm{MAlg}(X,\nu)$ sending $\left[ A \right]_\nu$ to $\left[ f^{-1}(A) \right]_\mu$ is a measure algebra morphism. Moreover, if $f$ is a bimeasurable bijection, then $\widetilde{f}$ is an isomorphism.

However, in general, given a morphism $\varphi \colon \mathrm{MAlg}(X,\nu) \rightarrow \mathrm{MAlg}(Y,\mu)$ there is no way to get a lifting of $\varphi$, that is a point to point measure-preserving map $\boldsymbol{\varphi} \colon Y \rightarrow X$ such that $\widetilde{\boldsymbol{\varphi}} = \varphi$. However, in the case of standard probability spaces, such a construction exists:

\begin{proposition}[{\cite[425D]{fremlinMeasureTheoryVol2013}}] \label{liftings}
Let $(X,\mu)$ and $(Y,\nu)$ be standard probability spaces. For every morphism of measure algebras $\varphi \colon \mathrm{MAlg}(X,\mu) \rightarrow \mathrm{MAlg}(Y,\nu)$ there is a lifting $\boldsymbol{\varphi} \colon Y \rightarrow X$ of $\varphi$. Moreover, for $\Gamma$ a countable group acting by automorphisms on $\mathrm{MAlg}(X,\mu)$ by an action $\alpha$, there is a lifting of $\alpha$, that is an action $\Gamma \act{\boldsymbol{\alpha}} X$ acting by measure-preserving transformations such that $\forall \gamma \in \Gamma,\ \widetilde{\gamma^{\boldsymbol{\alpha}}} = \left( \gamma^{-1}\right)^\alpha$.
\end{proposition}

\subsection{Model theory of atomless measure algebras}

We axiomatize the theory $\AMA$ of atomless measure algebras in the signature  $\pL = \{\vee,\wedge,\neg,0,1\}$ ($\bigtriangleup$ is defined as usual) as in \cite[Section 16]{yaacovModelTheoryMetric2008}.

\begin{proposition}[{\cite[16.2]{yaacovModelTheoryMetric2008}}] \label{AMA}
The theory $\AMA$ is separably categorical and therefore complete.
\end{proposition}

We also have:

\begin{proposition}[{\cite[16.6 and 16.7]{yaacovModelTheoryMetric2008}}]
The theory $\AMA$ admits quantifier elimination. Moreover, the definable closure $\mathrm{dcl}^\pM(C)$ of a subset $C$ in a model $\pM$ of $\AMA$ is the substructure $\left< C \right>$ of $\pM$ generated by $C$.
\end{proposition}

We will now give a characterization of the types in the theory $\AMA$. For that we need a little bit of terminology.

To any measure algebra $\pA$ we can associate a natural Hilbert space $\rL^2(\pA)$ called the \textbf{$\rL^2$ space of $\pA$}. This construction is consistent in the sense that if $\pA$ is the measure algebra of a probability space $(X,\mu)$, then there is a natural linear isometry between $\rL^2(\pA)$ and $\rL^2(X,\mu)$.

\begin{definition}
Let $\pA$ be a measure algebra and $\pB$ a measure subalgebra of $\pA$. Then the space $\rL^2(\pB)$ is a closed vector subspace of the Hilbert space $\rL^2(\pA)$, we denote by $\cP_\pB$ the orthogonal projection on $\rL^2(\pB)$ and we call it the \textbf{conditional expectation} with respect to $\pB$.
Particularly, for $a \in \pA$, $a$ can be seen as the element $\1_a$ of $\rL^2(\pA)$ and we call $\cP_\pB(\1_a)$ the \textbf{conditional probability} of $a$ with respect to $\pB$. For simplicity, we will denote it by $\cP_\pB(a)$.
\end{definition}

By definition, the conditional probability of $a$ with respect to $\pB$ is the only $\pB$-measurable function such that for any $\pB$-measurable function $f$, we have $\int \cP_\pB(a) f = \int \1_a f$.

\begin{proposition}[{\cite[16.5]{yaacovModelTheoryMetric2008}}] \label{types in AMA}
Let $\pM \models \AMA$, $\bar{a},\bar{b}$ be $n$-uples of elements of $\pM$ and  $C \subseteq \pM$. Then $\mathrm{tp}(\bar{a}/C) = \mathrm{tp}(\bar{b}/C)$ if and only if for every map $\sigma \colon \{1,\dots,n\} \rightarrow \{-1,1\}$ we have
$$\cP_{\left< C \right>} \left( \underset{1 \leq i \leq n}{\bigwedge} a_i^{\sigma(i)} \right) = \cP_{\left< C \right>} \left( \underset{1 \leq i \leq n}{\bigwedge} b_i^{\sigma(i)} \right),$$
where $a^1$ denotes $a$ and $a^{-1}$ denotes its complement $\neg\ a$ in $\pM$.
\end{proposition}

\subsection{The theory ${\fA}_\theta$}

Until now, we studied actions of any countable group. Since any action of a countable group can be represented as an $F_\infty$-action, for the sake of simplicity, we now restrict to $F_\infty$-actions, where $F_\infty$ denotes the countably generated free group.

We now expand the signature $\pL$ with a countable set of function symbols indexed by $F_\infty$, that we idendify with $F_\infty$ itself. We call this new signature $\pL_\infty$. We begin by considering the theory $\fA_{F_\infty}$ consisting of the following axioms:
\begin{itemize}
    \item The axioms of $\AMA$.
    \item For $\gamma \in F_\infty$, the axioms expressing that $\gamma$ is a measure algebra isomorphism:
    \begin{itemize}
        \item $\sup_{a,b}\ d(\gamma (a \vee b),\gamma a \vee \gamma b) = 0$
        \item $\sup_{a,b}\ d(\gamma (a \wedge b),\gamma a \wedge \gamma b) = 0$
        \item $\sup_a\ |\mu(\gamma a) - \mu(a)| = 0$
        \item $\sup_a\ \inf_b\ d(a, \gamma b) = 0$
    \end{itemize}
    \item The axioms expressing that $F_\infty$ acts on the measure algebra:
    \begin{itemize}
        \item $\sup_a\ d(1_{F_\infty} a,a) = 0$
        \item For $\gamma_1,\gamma_2 \in F_\infty$, the axiom $\sup_a\ d(\gamma_1 (\gamma_2 a), (\gamma_1  \gamma_2) a) = 0$
    \end{itemize}
\end{itemize}

By Propositions \ref{measure algebras arise from probability spaces} and \ref{liftings} any separable model of $\fA_{F_\infty}$ can be seen as the action on a measure algebra associated with a measure-preserving action $F_\infty \act{} (X,\mu)$ on a standard probability space.
If $\alpha$ is a pmp action on a probability space, we write $\pM_\alpha$ for the model of $\fA_{F_\infty}$ induced by $\alpha$. Without loss of generality, from now on, separable models we consider are always of the form $\pM_\alpha$ for $\alpha$ a pmp action on a standard probability space.

\begin{definition}
For $f$ any measure-preserving transformation $(X,\mu) \rightarrow (X,\mu)$, where $(X,\mu)$ is a probability space, we call the set $\{x \in X : fx \neq x\}$ the \textbf{support} of $f$ and we denote it by $\mathrm{Supp}\ f$.
\end{definition}

\begin{definition} \label{support}
Let $(\pA,\mu)$ be a measure algebra, the \textbf{support} of an automorphism $\varphi$ of $\pA$ is defined by $\mathrm{supp}\ \varphi = \bigwedge \{a \in \pA : \forall b \subseteq \neg a,\ \varphi b = b\}$.
\end{definition}

It is classic that if $f$ is a measure-preserving transformation of a standard probability space $(X,\mu)$, then $[\mathrm{Supp}\ f]_\mu = \mathrm{supp}\ \widetilde{f}$.

Our goal is now to give a first order description of the support of an automorphism of a separable measure algebra:

\begin{lemma} \label{support equivalence}
\begin{enumerate}
\item Let $\varphi$ be an automorphism of a separable atomless measure algebra $\pA$ such that $\mathrm{supp}\ \varphi \neq 0$. Then there exists $b \neq 0 \in \pA$ such that $\varphi b \wedge b = 0$.
\item Let $\pA$ be a separable atomless measure algebra. Let $\varphi$ be an automorphism of $\pA$.

Then there is $a_0 \in \pA$ such that $\mathrm{supp}\ \varphi = \varphi^{-1} a_0 \vee a_0 \vee \varphi a_0$ and $a_0 \wedge \varphi a_0 = 0$. Furthermore, we have $\mathrm{supp}\ \varphi = \bigvee \{\varphi^{-1} a \vee a \vee \varphi a : a \in \pA, a \wedge \varphi a = 0\}$.
\end{enumerate}
\end{lemma}

\begin{proof}
\begin{enumerate}
\item Consider a standard probability space $(X,\mu)$ such that $\mathrm{MAlg}(X,\mu) = \pA$ and let $f$ be a Borel lifting of $\varphi$ to $X$. Since $X$ is standard, let $(B_n : n \in \N)$ be a countable family of Borel subsets of $X$ separating the points. Without loss of generality, we may suppose that the set ${B_n : n \in \N}$ is stable by the operation of complement. For $n \in \N$, let $B_n' = B_n \setminus f^{-1}(B_n)$. For $x \in \mathrm{Supp}\ f$, there is $n$ such that $x \in B_n$ and $f(x) \notin B_n$ so $x \in B_n'$ and therefore $\mu(\underset{n \in \N}{\bigcup} B_n') \geq \mu(\mathrm{Supp}\ f) >0$. Take any $n$ such that $B_n'$ is of positive measure and let $b = [B_n']_\mu$.
\item First $\pA$ is a measure algebra and therefore is complete as a Boolean algebra so it has a maximal element $a_0$ disjoint from its image by $\varphi$.

Consider $b = \varphi^2 a_0 \setminus (\varphi^{-1} a_0 \vee a_0 \vee \varphi a_0)$. We have
\begin{eqnarray*}
(a_0 \vee b) \wedge \varphi(a_0 \vee b) &=& (a_0 \wedge \varphi a_0) \vee (a_0 \wedge \varphi b) \vee (b \wedge \varphi a_0) \vee (b \wedge \varphi b) \\
                            &\subseteq& 0 \vee (a_0 \setminus a_0) \vee (\varphi a_0 \setminus \varphi a_0 \vee (\varphi^2 a_0 \setminus \varphi^2 a_0) \\
                            &=& 0.
\end{eqnarray*}
Thus $a_0 \vee b$ is disjoint from its image. By maximality of $a_0$, we then have $b \subseteq a_0$, but by definition $b \wedge a_0 = 0$, so $b = 0$, or in other words, $\varphi^2 a_0 \subseteq \varphi^{-1} a_0 \vee a_0 \vee \varphi a_0$.

It follows that $\varphi \left( \varphi^{-1}a_0 \vee a_0 \vee \varphi a_0 \right) \subseteq \varphi^{-1}a_0 \vee a_0 \vee \varphi a_0$ and since $\varphi$ preserves the measure, the set $\varphi^{-1} a_0 \vee a_0 \vee \varphi a_0$ is invariant by $\varphi$.

Furthermore, $a_0$ is disjoint from its image by $\varphi$, and so $\varphi^{-1} a_0$ and $\varphi a_0$ are also disjoint from their respective image, so we have
$$\varphi^{-1} a_0 \vee a_0 \vee \varphi a_0 \subseteq \mathrm{supp}\ \varphi.$$

Conversely, let $c = \mathrm{supp}\ \varphi \setminus (\varphi^{-1} a_0 \vee a_0 \vee \varphi a_0)$ and suppose that $c \neq 0$. Since $c$ is invariant by $\varphi$, we can consider the automorphism $\varphi \restriction _{c}$ of the measure algebra lying under $c$. Applying the first point of this lemma to this automorphism, we get a non trivial $b \subseteq c$ disjoint from its image by $\varphi$.

But then, $a_0 \vee b$ contradicts the maximality of $a_0$. We conclude that
$$\varphi^{-1} a_0 \vee a_0 \vee \varphi a_0 = \mathrm{supp}\ \varphi.$$

Finally, as we already noticed, any set of the form $\varphi^{-1}a \vee a \vee \varphi a$ for $a \wedge \varphi a = 0$ is a subset of $\mathrm{supp}\ \varphi$, so we have
$$\mathrm{supp}\ \varphi = \bigvee \{\varphi^{-1} a \vee a \vee \varphi a : a \in \pA, a \wedge \varphi a =0\}.$$
\end{enumerate}
\end{proof}

Now we can prove that the IRS of a pmp action on a measure algebra is determined by the theory of this action seen as a model of $\fA_{F_\infty}$.

\begin{definition}
For $\gamma \in F_\infty$ we let $t_\gamma(a)$ denote the term $\gamma^{-1} (a \setminus \gamma a) \vee (a \setminus \gamma a) \vee \gamma (a \setminus \gamma a)$.

It follows from Lemma \ref{support equivalence} that for $\pM \models \fA_{F_\infty}$, $\mathrm{supp}\ \gamma = \bigvee \{t_\gamma(a) : a \in \pM\}$.
\end{definition}

\begin{lemma} \label{support definability}
Let $\gamma \in F_\infty$. Then the support of $\gamma$ is definable without parameters in the theory $\fA_{F_\infty}$.
\end{lemma}

\begin{proof}
We need to prove that the distance to $\mathrm{supp}\ \gamma$ is definable. By definition of the distance, we have $\forall a \in \pM,\ d(a,\mathrm{supp}\ \gamma) = \mu(a \setminus \mathrm{supp}\ \gamma) + \mu(\mathrm{supp}\ \gamma \setminus a)$.

On the one hand, $\mu(a \setminus \mathrm{supp}\ \gamma) = \inf_b\ \mu(a \setminus t_\gamma(b))$ so the first part is definable.

On the other hand, $\mu(\mathrm{supp}\ \gamma \setminus a) = \sup_b\ \mu(t_\gamma(b) \setminus a)$ and therefore the second part is definable as well.
\end{proof}

\begin{theorem} \label{irs is in theory}
Let $\pM_\alpha,\pM_\beta$ be two elementarily equivalent models of $\fA_{F_\infty}$. Then $\theta_\alpha = \theta_\beta$.
\end{theorem}

\begin{proof}
As $\theta_\alpha$ and $\theta_\beta$ are measures on $\mathrm{Sub}(F_\infty)$, they are determined by their values on the sets $\rN_{F,G} = \{\Lambda \leq F_\infty : F \subseteq \Lambda, G \cap \Lambda = \varnothing\}$ where $F$ and $G$ are finite.

Note that $\theta_\alpha(\rN_{F,\varnothing}) = \mu(\underset{\gamma \in F}{\bigcap} \mathrm{Supp}\ \gamma^\alpha)$ and $\theta_\beta(\rN_{F,\varnothing}) = \mu(\underset{\gamma \in F}{\bigcap} \mathrm{Supp}\ \gamma^\beta)$, but by Lemma \ref{support equivalence} these supports are the same as those defined in the measure algebra. Furthermore, by Lemma \ref{support definability}, for each $\gamma \in F_\infty$, $\mathrm{supp}\ \gamma$ is definable over $\varnothing$ in the theory $F_\infty$, and since the definable closure is a substructure, then $\underset{\gamma \in F}{\bigwedge} \mathrm{supp}\ \gamma$ must be definable over $\varnothing$ as well. Thus by elementary equivalence, for every finite $F \subseteq F_\infty$, we have $\theta_\alpha(\rN_{F,\varnothing}) = \theta_\beta(\rN_{F,\varnothing})$.

Now for $F,G$ finite subsets of $F_\infty$, write $\rN_{F,G} = \rN_{F,\varnothing} \setminus \underset{\gamma \in G}{\bigcup} \rN_{F \cup \{\gamma\},\varnothing}$. By the inclusion-exclusion principle, we then get
\begin{eqnarray*}
\theta_\alpha(\rN_{F,G}) &=& \theta_\alpha(\rN_{F,\varnothing}) + \sum_{i=1}^{|G|} (-1)^i \underset{\{J \subseteq G\ :\ |J|=i\}}{\sum} \theta_\alpha(\rN_{F \cup J,\varnothing}) \\
                       &=& \theta_\beta(\rN_{F,\varnothing}) +
\sum_{i=1}^{|G|} (-1)^i \underset{\{J \subseteq G\ :\
|J|=i\}}{\sum} \theta_\beta(\rN_{F \cup J,\varnothing}) \\
                       &=& \theta_\beta(\rN_{F,G}).
\end{eqnarray*}
\end{proof}

For $\theta$ an IRS, let $\fA_\theta$ be the $\pL_\infty$-theory consisting of:
\begin{itemize}
    \item The axioms of $\fA_{F_\infty}$.
    \item For $F \subseteq F_\infty$ finite, the axiom $\sup_{\{a_\gamma : \gamma \in F\}}\ \mu(\underset{\gamma \in F}{\bigwedge} t_\gamma(a_\gamma)) = \theta(\rN_{F,\varnothing})$.
\end{itemize}

Then the models of $\fA_\theta$ are exactly the measure-preserving actions of $F_\infty$ which have IRS $\theta$.

\subsection{Completeness and Model Completeness}

\begin{definition}
Let $(X,\mu)$ be a standard probability space and $\Gamma$ be a countable group.

First, let $\mathrm{Aut}(X,\mu)$ be the space of automorphisms of $\mathrm{MAlg}(X,\mu)$. We equip it with a complete metric $d_u$ called the \textbf{uniform metric} and defined by the formula $d_u(\varphi,\psi) \coloneqq \sup_{a \in \mathrm{MAlg}(X,\mu)} d_\mu(\varphi a,\psi a)$. We call the topology induced the \textbf{uniform topology}.

Then we define the space $A(\Gamma,X,\mu)$ of pmp actions of $\Gamma$ on $(X,\mu)$ naturally as a subspace of $\mathrm{Aut}(X,\mu)^\Gamma$. The uniform topology on $\mathrm{Aut}(X,\mu)$ gives rise to a product topology on $\mathrm{Aut}(X,\mu)^\Gamma$ which is completely metrizable and for which $A(\Gamma,X,\mu)$ is closed. Again, we call this topology the \textbf{uniform topology} on $A(\Gamma,X,\mu)$.
\end{definition}

From now on, fix a complete metric $d_u$ compatible with the uniform topology on $A(F_\infty,X,\mu)$.

\begin{theorem} \label{continuity of formulas}
Let $\varphi(\bar{x},\bar{y})$ be an $\pL_\infty$-formula, where $|\bar{x}| =n$, $|\bar{y}| =m$, let $(X,\mu)$ be a standard probability space and let $\bar{p} \in \mathrm{MAlg}(X,\mu)^m$.

Then the map $\begin{array}{ccccc}
\left( A(F_\infty,X,\mu), d_u \right) & \longrightarrow & \left( l^\infty (\mathrm{MAlg}(X,\mu)^n),\Vert\ \Vert_\infty \right) \\
\alpha & \longmapsto & \left( \varphi^{\pM_\alpha}(\bar{a},\bar{p}) \right)_{\bar{a} \in \mathrm{MAlg}(X,\mu)^n}
\end{array}$ is uniformly continuous.
\end{theorem}

\begin{proof}
We prove this result by induction on formulas. For now assume that the theorem holds for atomic formulas. First remark that if the theorem holds for certain formulas, then it holds for any combination of these formulas constructed with the help of connectives, by using their uniform continuity. Then it suffices to treat the case of quantifiers to conclude. But it is immediate, since we use the norm $\Vert\ \Vert_\infty$.
\\

Let us now prove the theorem for atomic formulas. If $\varphi(\bar{x},\bar{y})$ is an atomic formula, then it is equivalent to a formula of the form $\varphi(\bar{x},\bar{y}) \coloneqq \mu(t(\gamma_1\bar{x},\dots, \gamma_l \bar{x},\gamma_1\bar{y},\dots,\gamma_l \bar{y})$ for an $\pL$-term $t$ and some $\gamma_1,\dots,\gamma_l \in F_\infty$. Let $\varepsilon >0$.

By definition of the terms, they are uniformly continuous and so there is $\delta >0$ such that for $\bar{z}$ and $\bar{z'} \in \mathrm{MAlg}(X,\mu)^{(n+m)l}$, if $d_\mu(\bar{z},\bar{z'}) < \delta$ then $d_\mu(t(\bar{z}),t(\bar{z'})) < \varepsilon$.

Now if $\alpha,\beta \in A(F_\infty,X,\mu)$ are sufficiently $d_u$-close, then for every $a \in \mathrm{MAlg}(X,\mu)$ and $1 \leq i \leq l$, $d_\mu(\gamma_i^\alpha a,\gamma_i^\beta a) < \delta$. It follows that for all $\bar{a} \in \mathrm{MAlg}(X,\mu)^n$,
$$\left| \varphi^{\pM_\alpha}(\bar{a},\bar{p}) - \varphi^{\pM_\beta}(\bar{a},\bar{p}) \right| \leq d_\mu \left( t(\gamma_1^\alpha \bar{a},\dots,\gamma_l^\alpha \bar{a},\gamma_1^\alpha \bar{p},\dots,\gamma_l^\alpha \bar{p}), t(\gamma_1^\beta \bar{a},\dots,\gamma_l^\beta \bar{a},\gamma_1^\beta \bar{p},\dots,\gamma_l^\beta \bar{p}) \right) < \varepsilon,$$
which finishes the proof.
\end{proof}

\begin{theorem} \label{model completeness}
Let $\theta$ be a hyperfinite IRS on $F_\infty$. Then the theory ${\fA}_\theta$ is model complete.
\end{theorem}

\begin{proof}
It suffices to show that any inclusion of two separable models is elementary. Indeed, suppose this result and take any $\pM \subseteq \pN \models \fA_\theta$, $\varphi(\bar{x})$ a $\pL_\infty$-formula and $\bar{p} \in \pM$ finite. By the L\"owenheim-Skolem theorem, find a separable $\pM' \preceq \pM$ containing $\bar{p}$. Again by the L\"owenheim-Skolem theorem, find a separable $\pN' \preceq \pN$ containing the separable structure $\pM'$. Using the hypothesis, $\pM' \preceq \pN'$ so we finally get
$$\varphi(\bar{p})^\pM = \varphi(\bar{p})^{\pM'} = \varphi(\bar{p})^{\pN'} = \varphi(\bar{p})^\pN.$$

Let $\pM \subseteq \pN$ be two separable models of $\fA_\theta$. Consider a $\pL_\infty$-formula $\varphi(\bar{x})$ with $k$ variables and $\bar{p} \in \mathrm{MAlg}(X,\mu)^k$.

A classical argument derived from Proposition \ref{liftings} allows us to chose two pmp actions $F_\infty \act{\alpha} (X,\mu)$ and $F_\infty \act{\beta} (Y,\nu)$ on standard probability spaces along with a pmp map $\pi : Y \rightarrow X$, such that $\pM \simeq \pM_\alpha$, $\pN \simeq \pM_\beta$, and $\pi$ is a lifting of the inclusion $\mathrm{MAlg}(X,\mu) \hookrightarrow \mathrm{MAlg}(Y,\nu)$, which is equivariant respectively to the actions $\alpha$ and $\beta$. For $1 \leq i \leq k$, let $A_i \subseteq X$ be a Borel representative of $p_i$ and let $B_i = \pi^{-1}(A_i)$, which is also a Borel representative of $p_i$, in $Y$.

Then by Theorem \ref{approximate parametrized conjugacy for factor actions}, $\alpha$ is in the uniform closure of the set
$$\mC(\beta) \coloneqq \{\rho^{-1} \beta \rho : \rho\ \mbox{is a pmp bijection}\ X \rightarrow Y\ \mbox{such that}\ \forall i \leq k, \rho^{-1}(A_i) = B_i\}.$$
But then Theorem \ref{continuity of formulas} implies that $\varphi^{\pM_\alpha}(\bar{p}) \in \overline{\{\varphi^{\pM_{\beta'}}(\bar{p}) : \beta' \in \mC(\beta)\}}$.
Furthermore, for any $\beta' \in \mC(\beta)$, we have $(\beta',\bar{A}) \simeq (\beta,\bar{B})$, so that $(\pM_{\beta'},\bar{p}) \equiv (\pM_\beta,\bar{p})$ and consequently $\varphi^{\pM_{\beta'}}(\bar{p}) = \varphi^{\pM_{\beta}}(\bar{p})$. This establishes that $\varphi^{\pM_\alpha}(\bar{p}) = \varphi^{\pM_\beta}(\bar{p})$.

Hence $\pM_\alpha \preceq \pM_\beta$ and therefore $\fA_\theta$ is model complete.
\end{proof}

Now for completeness we combine model completeness with the argument of amalgamation already seen in Section \ref{amalgamation of actions}.

\begin{theorem}
Let $\theta$ be a hyperfinite IRS on $F_\infty$. Then the theory ${\fA}_\theta$ is complete.
\end{theorem}

\begin{proof}
As usual, it is sufficient to prove that two separable models of $\fA_\theta$ are elementarily equivalent.

Let $\pM_\alpha, \pM_\beta \models \fA_\theta$ be two separable models and consider the action $\zeta \coloneqq \alpha \underset{\boldsymbol{\theta}}{\times} \beta$. By Lemma \ref{joining over an irs}, we have $\pM_\zeta \models \fA_\theta$ and moreover, both $\pM_\alpha$ and $\pM_\beta$ are substructures of $\pM_\zeta$.

Now since $\fA_\theta$ is model complete, we have $\pM_\alpha \preceq \pM_\zeta$ and $\pM_\beta \preceq \pM_\zeta$, so $\pM_\alpha \equiv \pM_\zeta \equiv \pM_\beta$.
\end{proof}

\subsection{Elimination of quantifiers}

\begin{proposition}[{\cite[Prop. 13.16]{yaacovModelTheoryMetric2008}}] \label{characterization of quantifier elimination}
Let $T$ be a countable theory. Then $T$ admits quantifier  elimination if and only if for any $\pM,\pN \models T$, any substructure $\pA \subseteq \pM$ and any embedding $f \colon \pZ \hookrightarrow \pN$, there is an elementary extension $\pN'$ of $\pN$ and an embedding $\tilde{f} \colon \pM \hookrightarrow \pN'$ extending $f$.
\end{proposition}

\begin{definition} \label{amalgamation}
We say that a theory $T$ admits \textbf{amalgamation} if for any $\pM_1,\pM_2 \models T$ and any common substructure $\pZ$, there is a joining of $\pM_1$ and $\pM_2$ over $\pZ$, that is a structure $\pN \models T$ and embeddings $\pM_i \hookrightarrow \pN$ ($i=1,2$) such that the following diagram commutes:

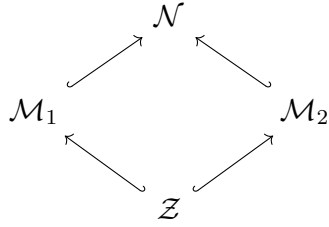
\begin{center}
\begin{tikzcd}
& \pN & \\
\pM_1 \ar[ru,hook] & & \ar[lu,hook'] \pM_2 \\
& \ar[lu,hook'] \pZ \ar[ru,hook] &
\end{tikzcd}
\end{center}
\end{definition}

The next lemma is a classical result in discrete model theory and it easily extends to continuous model theory.

\begin{lemma} \label{model completeness + amalgamation = quantifier elimination}
Let $T$ be a theory. Then $T$ admits quantifier elimination if and only if it admits amalgamation and is model complete.
\end{lemma}

\begin{proof}
Suppose that $T$ admits quantifier elimination. Let $\pM_1,\pM_2 \models T$ with a common substructure $\pZ$, applying Proposition \ref{characterization of quantifier elimination} where $f$ is the inclusion $\pZ \hookrightarrow \pM_2$, we get $\pN$ as required.

Now let $\pM \subseteq \pN$ be two models of $T$. By quantifier elimination, we only need to prove that $\pM \models \varphi(\bar a) \Leftrightarrow \pN \models \varphi(\bar a)$ for atomic formulas $\varphi$ and finite tuples $\bar a$ of parameters in $\pM$. But this is trivial by the definition of inclusion for models.
\\

Conversely, suppose $T$ admits amalgamation and is model complete and let $\pM,\pN \models T$, $\pZ \subseteq \pM$ be a substructure, and $f \colon \pZ \hookrightarrow \pN$. By considering a monster model, we may suppose that $\pZ \subseteq \pN$ and $f$ is the identity. Then by amalgamation there is a model $\pN' \models T$ and embeddings $\varphi,\psi$ such that the following diagram commutes:

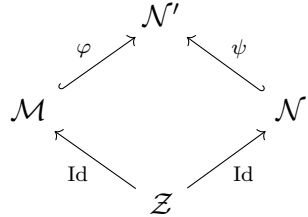
\begin{center}
\begin{tikzcd}
& \pN' & \\
\pM \ar[ru,hook,"\varphi"] & & \ar[lu,hook',"\psi"'] \pN \\
& \ar[lu,"\mathrm{Id}"] \pZ \ar[ru,"\mathrm{Id}"'] &
\end{tikzcd}
\end{center}
Again we may suppose that $\pN \subseteq \pN'$ and $\psi$ is the identity, thus by model completeness we have $\pN \preceq \pN'$. Furthermore, the diagram now exactly states that $\varphi$ extends the inclusion $\pZ \hookrightarrow \pN$.
\end{proof}

In order to prove that our theories eliminate quantifiers, it only remains to prove that they have amalgamation. However, the following example shows that this is not the case in general.

\begin{definition}
Let $\Gamma \act{\alpha} X$ be an action of a group on a standard Borel space. We say that $\mu \in \fP(X)$ is \textbf{ergodic} if every $\Gamma$-invariant for $\alpha$ measurable subset of $X$ is either null or connull for $\mu$.

It can be shown that ergodic measures are the extreme points of the convex space $\fP(X)$.
\end{definition}

For Invariant Random Subgroups, we consider the notion of ergodicity with respect to the action $\Gamma \act{} \mathrm{Sub}(\Gamma)$ by conjugation.

\begin{proposition} \label{not quantifier elimination}
Let $\theta$ be a non-ergodic IRS on $F_\infty$. Then ${\fA}_\theta$ does not have quantifier elimination.
\end{proposition}

\begin{proof}
Take any finite subset $F \subseteq F_\infty$. Then $\mu \left(x \wedge \underset{\gamma \in F}{\bigwedge}\mathrm{supp}\ \gamma \right) \coloneqq \sup_{\{a_\gamma : \gamma \in F\}}\ \mu \left(x \wedge \underset{\gamma \in F}{\bigwedge} t_\gamma(a_\gamma) \right)$ is a definable predicate in the signature $\pL_\infty$.
However, as we shall see, not all predicates of this form are definable without quantifiers.

Indeed, suppose that for every finite subset $F \subseteq F_\infty$, there is a quantifier free formula $\varphi_F(x)$ equivalent to $\mu \left(x \wedge \underset{\gamma \in F}{\bigwedge}\mathrm{supp}\ \gamma \right)$.

Write $\theta = t \theta_1 + (1-t) \theta_2$ for a $t \in (0,\frac{1}{2}]$ and $\theta_1 \neq \theta_2$ two IRSs on $F_\infty$. Let $\kappa_1$ be a pmp action on $([0,1],\lambda)$ with IRS $\theta_1$ and $\kappa_2$ be a pmp action on $([0,1],\lambda)$ with IRS $\theta_2$. Define
\begin{itemize}
    \item $F_\infty \act{\alpha} (X = [0,1] \times \{1,2,3\}, \mu = t \lambda \times \delta_1 + t \lambda \times \delta_2 + (1-2t) \lambda \times \delta_3)$ that acts like $\kappa_1$ on $[0,1] \times \{1\}$ and acts like $\kappa_2$ both on $[0,1] \times \{2\}$ and on $[0,1] \times \{3\}$.
    \item $F_\infty \act{\beta} (X = [0,1] \times \{1,2,3\}, \mu = t \lambda \times \delta_1 + t \lambda \times \delta_2 + (1-2t) \lambda \times \delta_3)$ that acts like $\kappa_1$ on $[0,1] \times \{2\}$ and acts like $\kappa_2$ both on $[0,1] \times \{1\}$ and on $[0,1] \times \{3\}$.
\end{itemize}
We have $\theta_\alpha = \theta_\beta = \theta$.

Let $\pM$ be the finite measure algebra generated by three atoms $\{a,b,c\}$ of respective measure $t$, $t$ and $1-2t$. By sending $a$ to $[0,1] \times \{1\}$, $b$ to $[0,1] \times \{2\}$ and $c$ to $[0,1] \times \{3\}$, one can embed $\pM$ in both $\pM_\alpha$ and $\pM_\beta$. Then $\pM$ endowed with the trivial action is a common substructure of $\pM_\alpha$ and $\pM_\beta$.

As $\varphi_F(x)$ is quantifier free, we have $\varphi_F^{\pM_\alpha}(a) = \varphi_F^\pM(a) = \varphi_F^{\pM_\beta}(a)$, but
$$\pM_\alpha \models \mu(a \wedge \underset{\gamma \in F}{\bigwedge}\mathrm{supp}\ \gamma) = t \theta_1(\rN_F) \ \mbox{whereas}\ \pM_\beta \models \mu(a \wedge \underset{\gamma \in F}{\bigwedge}\mathrm{supp}\ \gamma) = t \theta_2(\rN_F).$$
Since an IRS is determined by its values on the sets of the form $\rN_F$, we get $\theta_1 = \theta_2$, a contradiction.
\end{proof}

Thus, non-ergodicity of the IRS is an obstacle to quantifier elimination. A natural question is to ask about a converse:

\begin{center}
\textit{For which $\theta$ does the theory $\fA_\theta$ admit quantifier elimination? Is it the case for any ergodic IRS?}
\end{center}

The author does not have any satisfying answer.

However, we answer another interesting question. One can ask what we can reasonably add to the theory $\fA_\theta$ to expand it into a theory $\fA_\theta'$ in a signature $\pL_\infty' \supseteq \pL_\infty$ which has quantifier elimination.

The issue encountered in Proposition \ref{not quantifier elimination} is that formulas involving the supports of the elements of $F_\infty$ may not be equivalent to quantifiers free formulas in $\fA_\theta$.
This motivates us to look at expansions that allow us to talk about the supports of elements of $F_\infty$ in the language. For that we add constants $\{S_\gamma : \gamma \in F_\infty\}$ to the signature $\pL_\infty$ to get a new signature $\pL_\infty'$ and we consider the theory $\fA_\theta'$ consisting of:
\begin{itemize}
    \item The axioms of $\fA_\theta$.
    \item For $\gamma \in F_\infty$, the axioms:
    \begin{itemize}
        \item $\sup_a\ d(S_\gamma \wedge t_\gamma(a),t_\gamma(a)) = 0$.
        \item $\mu(S_\gamma) = \theta(\rN_\gamma)$.
    \end{itemize}
\end{itemize}

This theory expresses that for $\gamma \in F_\infty$, the constant $S_\gamma$ must be interpreted as $\mathrm{supp}\ \gamma^\pM$ in the model $\pM$, as it contains the support by the first axiom and has the same measure by the second one.
\\

We need a last definition in order to prove that the theories $\fA_\theta$ admit amalgamation for $\theta$ hyperfinite:

\begin{definition}
Let $\pM \models \fA_\theta$, we denote by $\pI_\pM$ and we call the IRS of $\pM$ the substructure of $\pM$ generated by the elements $\mathrm{supp}\ \gamma$ for $\gamma \in \Gamma$.

Note that this naming is consistent: let $\pM = \pM_\alpha$ for a pmp action $\Gamma \act{\alpha} (X,\mu)$ of IRS $\theta$. Then $\pI_\pM$ is isomorphic to the measure algebra $\pI_\theta$ associated to the action $\Gamma \act{\boldsymbol{\theta}} (\mathrm{Sub}(\Gamma),\theta)$ and moreover, the map $\mathrm{Stab}^\alpha \colon X \rightarrow \mathrm{Sub}(\Gamma)$ is a lifting of the inclusion $\pI_\pM \subseteq \pM$.
\end{definition}

\begin{theorem} \label{relative independent joining of stone spaces}
Let $\theta$ be an IRS, then the theory $\fA_\theta'$ admits amalgamation in the signature $\pL_\infty'$.
\end{theorem}

\begin{proof}
Let $\pM_1, \pM_2 \models \fA_\theta'$ and let $\pZ$ be a common substructure of $\pM_1$ and $\pM_2$. Then by definition of the theory $\fA_\theta'$, $\pI_\theta$ is a substructure of $\pZ$ and the inclusions $\pZ \hookrightarrow \pM_1$ and $\pZ \hookrightarrow \pM_2$ send $\pI_\theta$ on $\pI_{\pM_1}$ and $\pI_{\pM_2}$ respectively. For the sake of simplicity, we identify $\pZ$ with its images in $\pM_1$ and $\pM_2$, which implies that $\pI_\theta$, $\pI_{\pM_1}$ and $\pI_{\pM_2}$ are all identified.

Let $X_1$, $X_2$ and $Z$ be the respective Stone spaces of $\pM_1$, $\pM_2$ and $\pZ$ (see \cite[321J]{fremlinMeasureTheoryVol2002}) and let $\mu_1$, $\mu_2$ be the respective inner regular Borel probability measures on $X_1$ and $X_2$. We define an inner regular Borel probability measure $\nu$ on $X_1 \times X_2$ as in \cite[Construction 2.3]{SchrodingerCat} as the continuous extension of the map defined on cylinders by the formula:
$$\nu(a_1 \times a_2) = \int_Z\ \mu_1(a_1|\pZ) \mu_2(a_2|\pZ)\ dz\ \ \mbox{for all}\ a_1 \in \pM_1, a_2 \in \pM_2.$$

The pmp action $F_\infty \act{} (X_1 \times X_2,\nu)$ then induces a structure $\pN \models \fA_{F_\infty}$ that we call the {\bfseries relative independent joining of $\pM_1$ and $\pM_2$ over $\pZ$}.

The following diagram is indeed commutative:
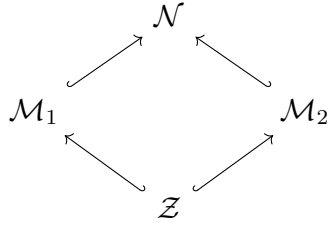
\begin{center}
\begin{tikzcd}
& \pN & \\
\pM_1 \ar[ru,hook] & & \ar[lu,hook'] \pM_2 \\
& \ar[lu,hook'] \pZ \ar[ru,hook] &
\end{tikzcd}
\end{center}
It remains to prove that $\pN \models \fA_\theta$. For that note that
\begin{eqnarray*}
\neg\ \mathrm{supp}\ \gamma^\pN &=& \bigvee \{a : \forall b \subseteq a, \gamma b = b\}\\
&=& \bigvee \{a_1 \times a_2 : \forall b \subseteq a_1 \times a_2, \gamma b = b\}\\
&=& \bigvee \{a_1 \times a_2 : \forall b_1 \subseteq a_1\ \forall b_2 \subseteq a_2, \gamma b_1 = b_1\ \mbox{and}\ \gamma b_2 = b_2\}\\
&=& \neg\ \mathrm{supp}\ \gamma^{\pM_1} \times \neg\ \mathrm{supp}\ \gamma^{\pM_2}\\
&=& \neg\ S_\gamma^\pZ \times \neg\ S_\gamma^\pZ
\end{eqnarray*}
but the definition of $\nu$ implies that $\nu \left( \neg\ S_\gamma^\pZ \times 1_{\pM_2} \right) = \nu \left( \neg\ S_\gamma^\pZ \times \neg\ S_\gamma^\pZ \right)$, so that these two elements of $\pN$ are equal. Letting $i_1$ denote the embedding $\pM_1 \hookrightarrow \pN$, we get the equalities $\neg\ \mathrm{supp}\ \gamma^\pN = \neg\ S_\gamma^\pZ$ and therefore $\mathrm{supp}\ \gamma^\pN = i_1 \left( S_\gamma^\pZ \right) = i_1 \left( \mathrm{supp}\ \gamma^{\pM_1} \right)$.
This being true for any $\gamma \in F_\infty$, it follows that $i_1$ maps any finite intersection of supports in $\pM_1$ to the corresponding intersection of supports in $\pN$, and since $i_1$ also preserves the measure, we can conclude that $\pN \models \fA_\theta$.
\end{proof}

\begin{theorem} \label{quantifier elimination}
Let $\theta$ be a hyperfinite IRS. Then the theory $\fA_\theta'$ eliminates quantifiers in the signature $\pL_\infty'$.
\end{theorem}

\begin{proof}
We use Lemma \ref{model completeness + amalgamation = quantifier elimination}.

We just saw that $\fA_\theta'$ admits amalgamation.

For model completeness, take $\pM \subseteq \pN$ be two models of $\fA_\theta'$ and let us prove that $\pM \preceq \pN$. Let $\varphi(\bar{x})$ be an $\pL_\infty'$-formula and $\bar{p} \in \pM^n$. Then $\varphi(\bar{x})$ is equivalent to a formula of the form $\psi(\bar{x},S_{\bar{\gamma}})$ where $\psi$ is a $\pL_\infty$-formula, and the constants of the form $S_\gamma$ are preserved under the inclusion $\pM \subseteq \pN$. Therefore, it suffices to apply Theorem \ref{model completeness} to $\psi$ and to consider the elements $S_{\bar{\gamma}}$ as parameters added to $\bar{p}$ to conclude.
\end{proof}

As a corollary, we get a class of IRSs $\theta$ for which the theory $\fA_\theta$ admits quantifier elimination.

\begin{corollary}
The theory of free actions of an amenable group admits amalgamation. Namely, if $\theta$ is the Dirac measure $\delta_N$ for a co-amenable normal subgroup $N \leq F_\infty$, then $\fA_\theta$ has quantifier elimination.
\end{corollary}

\begin{proof}
Simply note that the support of an element $\gamma \in F_\infty$ in a model of $\fA_\theta$ is either $0$ (if $\gamma \in N$) or $1$ (if $\gamma \notin N$). It follows that the theories $\fA_\theta$ and $\fA_\theta'$ completely coincide, hence the result.
\end{proof}

For $\pM \models \fA_{\infty}$ and $A \subseteq \pM$, we write $\left< A \right>$ for the closed subalgebra of $\pM$ (that is, the substructure of $\pM$ as a model of $\AMA$) generated by $A$.

\begin{theorem} \label{definable closure}
Let $\pM \models \fA_\theta$ and $A \subseteq \pM$. Then the definable closure of $A$ in $\pM$ is $\left< F_\infty A \cup \pI_\pM \right>$.
\end{theorem}

\begin{proof}
On the one hand, $A \subseteq \mathrm{dcl}^\pM(A)$ and by Lemma \ref{support definability}, for $\gamma \in F_\infty$, $\mathrm{supp}\ \gamma^\pM \in \mathrm{dcl}^\pM(A)$. Thus we get the first inclusion.

On the other hand, since $\fA_\theta'$ expands $\fA_\theta$, the definable closure of $A$ in the theory $\fA_\theta$ is contained in the definable closure of $A$ in the theory $\fA_\theta'$. Let us compute this definable closure $D$.

First, we notice that the function symbols $\gamma$ are interpreted by automorphisms and thus any atomic $\pL_\infty$-formula with parameters in $A$ is equivalent to an atomic $\pL$-formula with parameters in $F_\infty A$. This remark then extends to quantifier free formulas.

Then, by Theorem \ref{quantifier elimination}, any $\pL_\infty'$-formula with parameters in $A$ is equivalent to a quantifier free $\pL_\infty'$-formula with parameters in $A$ and since we only added constants in $\pL_\infty$, it is moreover equivalent to a quantifier free $\pL_\infty$-formula with parameters in $A \cup \pI_\pM$.

Combining the two latter properties and the fact that $\mathrm{dcl}(A) = \left< A \right >$ in the theory $\AMA$, we get that $D = \left< F_\infty (A \cup \pI_\pM) \right>$. Furthermore, $\pI_\pM$ is a substructure and so $\left< F_\infty (A \cup \pI_\pM) \right> = \left< F_\infty A \cup \pI_\pM \right>$.

Hence the conclusion.
\end{proof}

\subsection{Stability and Independence}

We recall some definitions from \cite{yaacovModelTheoryMetric2008}.

\begin{definition}
Let $\kappa$ be a cardinal. A \textbf{$\kappa$-universal domain} for a theory $T$ is a $\kappa$-saturated and strongly $\kappa$-homogeneous model of $T$. If $\pU$ is a $\kappa$-universal domain and $A \subseteq \pU$, we say that $A$ is \textbf{small} if $|A| < \kappa$.
\end{definition}

\begin{definition}
Let $\pU$ be a $\kappa$-universal domain for $T$. A \textbf{stable independence relation} on $\pU$ is a relation $A \underset{C}{\ind} B$ on triples of small subsets of $\pU$ satisfying the following properties, for all small $A,B,C,D \subseteq \pU$, finite $\bar{u}, \bar{v} \subseteq \pU$ and small $\pM \preceq \pU$ :
\begin{enumerate}
    \item \textit{Invariance under automorphisms of $\pU$}.
    \item \textit{Symmetry:} $A \underset{C}{\ind} B \Longleftrightarrow B \underset{C}{\ind} A$.
    \item \textit{Transitivity:} $A \underset{C}{\ind} BD \Longleftrightarrow A \underset{C}{\ind} B \wedge A \underset{BC}{\ind} D$.
    \item \textit{Finite character:} $A \underset{C}{\ind} B$ if and only if $\bar{a} \underset{C}{\ind} B$ for every finite $\bar{a} \subseteq A$.
    \item \textit{Existence:} There exists $A'$ such that $\mathrm{tp}(A'/C) = \mathrm{tp}(A/C)$ and $A' \underset{C}{\ind} B$.
    \item \textit{Local character:} There exists $B_0 \subseteq B$ such that $|B_0| \leq |T|$ and $\bar{u} \underset{B_0}{\ind} B$.
    \item \textit{Stationarity of types:} If $\mathrm{tp}(A/\pM) = \mathrm{tp}(B/\pM)$ and $A \underset{\pM}{\ind} C$ and $B \underset{\pM}{\ind} C$, then
    
$\mathrm{tp}(A/\pM \cup C)= \mathrm{tp}(B/\pM \cup C)$.
\end{enumerate}
\end{definition}

\begin{proposition}[\cite{yaacovModelTheoryMetric2008}]
Let $\kappa > |T|$ and let $\pU$ be a $\kappa$-universal domain. Then the theory $T$ is stable if and only if there exists a stable independence relation on $\pU$, and in this case the stable independence relation is the independence relation given by non-dividing.
\end{proposition}

Thus, in order to prove that our theories are stable, we only need to define a stable independence relation. Ben Yaacov proved in \cite[Thm. 4.1]{SchrodingerCat} that the classical relation of independence of events was the required one in the case of measure algebras without group actions. Now that we described the definable closures in our theories, the proof of Ben Yaacov naturally adapts to this case.

\begin{definition} \label{stability}
From now on, we write $\left< \left< A \right> \right>$ for $\mathrm{dcl}^\pU(A)$.

Let $A,B,C \subseteq \pU$, we say that $A$ and $B$ are independent over $C$ and we write $A \underset{C}{\ind} B$ if we have $\forall a \in \left< \left< A \right> \right>,\ \forall b \in \left< \left< B \right> \right>$, $\cP_{\left< \left< C \right> \right>}(a) \cP_{\left< \left< C \right> \right>}(b) = \cP_{\left< \left< C \right> \right>}(a \wedge b)$.
\end{definition}

We will need the following propositions:

\begin{proposition}[{\cite[Proposition 5.6]{FOMP}}] \label{independence lemma}
Let $A,B,C \subseteq \pU \models \fA_{F_\infty}$. Then we have $A \underset{C}{\ind} B$ if and only if  $\forall a \in \left< \left< A \right> \right>$,
$$\cP_{\left< \left< BC \right> \right>}(a) = \cP_{\left< \left< C \right> \right>}(a).$$
\end{proposition}

\begin{proposition}[{\cite[Lemma 2.7]{SchrodingerCat}}] \label{ben yaacov}
Let $\theta$ be a hyperfinite IRS on $F_\infty$.

Let $\pU \models \fA_\theta'$ and let $\pM_1, \pM_2$ be small substructures of $\pU$. Let $\pZ$ be a common substructure of $\pM_1$ and $\pM_2$. Let $\pM_1 \wedge \pM_2$ be the substructure of $\pU$ generated by $\pM_1$ and $\pM_2$ and define $\pN$ the relative independent joining of $\pM_1$ and $\pM_2$ over $\pZ$ as in Theorem \ref{relative independent joining of stone spaces}.

Then $\pM_1 \underset{\pZ}{\ind} \pM_2$ if and only if $\pM_1 \wedge \pM_2 \simeq \pN$.
\end{proposition}

\begin{theorem}
If $\theta$ is a hyperfinite IRS, the relation of independence $\ind$ defined above is a stable independence relation when restricted to triples of small subsets, relatively to the theory $\fA_\theta$. Consequently, the theory $\fA_\theta$ is stable and the relation $\ind$ agrees with non-dividing on triples of small subsets.
\end{theorem}

\begin{proof}
\begin{enumerate}
    \item \textit{Invariance under automorphisms of $\pU$:} If $\rho$ is an automorphism of $\pU$, by  uniqueness of the orthogonal projection, we know that $\cP_{\left< \left< \rho(C) \right> \right>} = \rho \circ \cP_{\left< \left< C \right> \right>} \circ \rho^{-1}$ and therefore
$$\cP_{\left< \left< C \right> \right>}(a) \cP_{\left< \left< C \right> \right>}(b) = \cP_{\left< \left< C \right> \right>}(a \wedge b) \Leftrightarrow \cP_{\left< \left< \rho(C) \right> \right>}(\rho a) \cP_{\left< \left< \rho(C) \right> \right>}(\rho b) = \cP_{\left< \left< \rho(C) \right> \right>}(\rho (a \wedge b)).$$
    
    \item \textit{Symmetry:} The definition is symmetric.
    
    \item \textit{Transitivity:} Let $A,B,C,D$ be small.
    First if $A \underset{C}{\ind} B$ and $A \underset{BC}{\ind} D$ then by Proposition \ref{independence lemma}, for $a \in \left< \left< A \right> \right>$, we have $\cP_{\left< \left< BCD \right> \right>}(a) = \cP_{\left< \left< BC \right> \right>}(a) = \cP_{\left< \left< C \right> \right>}(a)$ so $A \underset{C}{\ind} BD$.
    
    Conversely, suppose that $A \underset{C}{\ind} BD$. Then $\cP_{\left< \left< BCD \right> \right>}(a) = \cP_{\left< \left< C \right> \right>}(a)$, but that implies that $\cP_{\left< \left< C \right> \right>}(a)$ is a $\left< \left< C \right> \right>$-measurable function such that for all $\left< \left< BCD \right> \right>$-measurable function $f$ we have $\int \cP_{\left< \left< C \right> \right>}(a) f = \int \1_a f$. We conclude that $\cP_{\left< \left< BCD \right> \right>}(a) = \cP_{\left< \left< BC \right> \right>}(a) = \cP_{\left< \left< C \right> \right>}(a)$, and therefore that $A \underset{C}{\ind} C$ and $A \underset{BC}{\ind} D$.
    
    \item \textit{Finite character:} It follows from the definition and the continuity of $\cP$.
    
    \item \textit{Existence:} Let $A, B, C$ be small subsets of $\pU$. By L\"owenheim-Skolem theorem, let $\pA$ and $\pB$ be small structures such that $\left< \left< AC \right> \right> \subseteq \pA \preceq \pU$ and $\left< \left< BC \right> \right> \subseteq \pB \preceq \pU$, and let $\pC = \left< \left< C \right> \right>$. Then $\pA$ and $\pB$ are both elementary substructures of $\pU$ containing $\pI_\pU$. It follows that $\pA$ and $\pB \models \fA_\theta'$ when the constants $S_\gamma$ are interpreted by $\mathrm{supp}\ \gamma^\pU$ in either of these models, and $\pC$ is an $\pL_\infty'$-common substructure of $\pA$ and $\pB$, so using Theorem \ref{relative independent joining of stone spaces}, we see that the relative independent joining $\pD$ of $\pA$ and $\pB$ over $\pC$ is a small model of $\fA_\theta$.
    
    By saturation and homogeneity of $\pU$, we can embed $\pD$ in $\pU$ while sending $\pB$ back to $\pB$. Taking the image of $\pA$ by this embedding gives us a new copy $\pA'$ of $\pA$ and a new copy $A'$ of $A$. Finally, $\pA' \wedge \pB \simeq \pD$ so by Proposition \ref{ben yaacov} we get that $\pA' \underset{\pC}{\ind} \pB$, which in turn implies that $A' \underset{C}{\ind} B$.
    
    \item \textit{Local character:} Let $\bar{u} = (u_1,\dots,u_n) \subseteq \pU$ be finite. Consider the conditional probabilities $\cP_{\left< \left< B \right> \right>}(u_i)$. These are $\left< \left< B \right> \right>$-measurable functions with real values and so there is a countably generated $\sigma$-subalgebra of $\left< \left< B \right> \right>$, say $\left< \left< B_0 \right> \right>$ where $B_0 \subseteq B$ is countable, for which they are all measurable.
    But then we have $\cP_{\left< \left< B \right> \right>}(u_i) = \cP_{\left< \left< B_0 \right> \right>}(u_i)$, so by Proposition \ref{independence lemma} $\bar{u} \underset{B_0}{\ind} B$.
      
    \item \textit{Stationarity of types:} We denote by $\mathrm{tp}_\pL(\bar{x}/Y)$ the type of a tuple $\bar{x}$ over a set of parameters $Y$ in the language $\pL$. In other words, this is the type of $\bar{x}$ over $Y$ in the underlying atomless measure algebra of $\pU$.
    
    Let $A,B,C \subseteq \pU$ be small and $\pM \preceq \pU$ be small. Suppose that $\mathrm{tp}(A/\pM) = \mathrm{tp}(B/\pM)$, $A \underset{\pM}{\ind} C$ and $B \underset{\pM}{\ind} C$.
    
     We begin by proving that $\mathrm{tp}_\pL(A/ \left< \left< \pM \cup C \right> \right>) = \mathrm{tp}_\pL(B/ \left< \left< \pM \cup C \right> \right>)$. Indeed, for $a \in  \left< A \right>$ and $b \in \left< B \right>$, we have $\cP_{\left< \left< \pM \cup C \right> \right>}(a) = \cP_\pM(a)$ and $\cP_{\left< \left< \pM \cup C \right> \right>}(b) = \cP_\pM(b)$, but by Proposition \ref{types in AMA} types in $\AMA$ can be fully described with conditional probabilities and we have $\mathrm{tp}_\pL(A/\pM) = \mathrm{tp}_\pL(B/\pM)$ so we get $\mathrm{tp}_\pL(A/ \left< \left< \pM \cup C \right> \right>) = \mathrm{tp}_\pL(B/ \left< \left< \pM \cup C \right> \right>)$.  
    
    Now Theorem \ref{quantifier elimination} implies that $\mathrm{tp}(A/\pM \cup C)$ (resp. $\mathrm{tp}(B/\pM \cup C)$) is determined by the $\pL$-type $\mathrm{tp}_\pL \left( \left< F_\infty A \cup \pI_\pU \right> / \left< \left< \pM \cup C \right> \right> \right)$ (resp. $\mathrm{tp}_\pL \left( \left< F_\infty B \cup \pI_\pU \right> / \left< \left< \pM \cup C \right> \right> \right)$).
    
    Thus, let $A' = F_\infty A \cup \pI_\pU$ and $B' = F_\infty B \cup \pI_\pU$.
    
    It is clear that $\mathrm{tp}(A'/\pM) = \mathrm{tp}(B'/\pM)$, $A' \underset{\pM}{\ind} C$ and $B' \underset{\pM}{\ind} C$ and we can apply what we proved just above to conclude that $\mathrm{tp}_\pL \left( A' / \left< \left< \pM \cup C \right> \right> \right) = \mathrm{tp}_\pL \left( B' / \left< \left< \pM \cup C \right> \right> \right)$, that is
$$\mathrm{tp}_\pL \left( \left< F_\infty A \cup \pI_\pU \right> / \left< \left< \pM \cup C \right> \right> \right) = \mathrm{tp}_\pL \left( \left< F_\infty B \cup \pI_\pU \right> / \left< \left< \pM \cup C \right> \right> \right),$$
    
hence the conclusion.
\end{enumerate}
\end{proof}


\bibliography{Pmp_hyperfinite_actions}

\end{document}